\documentclass[12pt]{article}
\usepackage{amssymb,amsfonts,epic,graphics,pictex}
\begin{document}

\newtheorem{lem}{Lemma}[section]
\newtheorem{prop}{Proposition}
\newtheorem{con}{Construction}[section]
\newtheorem{defi}{Definition}[section]
\newtheorem{coro}{Corollary}[section]
\newcommand{\hf}{\hat{f}}
\newtheorem{fact}{Fact}[section]
\newtheorem{theo}{Theorem}
\newcommand{\Br}{\Poin}
\newcommand{\Cr}{{\bf Cr}}
\newcommand{\dist}{{\bf dist}}
\newcommand{\diam}{\mbox{diam}\, }
\newcommand{\mod}{{\rm mod}\,}
\newcommand{\compose}{\circ}
\newcommand{\dbar}{\bar{\partial}}
\newcommand{\Def}[1]{{{\em #1}}}
\newcommand{\dx}[1]{\frac{\partial #1}{\partial x}}
\newcommand{\dy}[1]{\frac{\partial #1}{\partial y}}
\newcommand{\Res}[2]{{#1}\raisebox{-.4ex}{$\left|\,_{#2}\right.$}}
\newcommand{\sgn}{{\rm sgn}}

\newcommand{\C}{{\bf C}}
\newcommand{\D}{{\bf D}}
\newcommand{\Dm}{{\bf D_-}}
\newcommand{\N}{{\bf N}}
\newcommand{\R}{{\bf R}}
\newcommand{\Z}{{\bf Z}}
\newcommand{\tr}{\mbox{Tr}\,}

\newenvironment{nproof}[1]{\trivlist\item[\hskip \labelsep{\bf Proof{#1}.}]}
{\begin{flushright} $\square$\end{flushright}\endtrivlist}
\newenvironment{proof}{\begin{nproof}{}}{\end{nproof}}

\newenvironment{block}[1]{\trivlist\item[\hskip \labelsep{{#1}.}]}{\endtrivlist}
\newenvironment{definition}{\begin{block}{\bf Definition}}{\end{block}}

\newtheorem{conjec}{Conjecture}
\newtheorem{com}{Comment}
\newtheorem{exa}{Example}
\font\mathfonta=msam10 at 11pt
\font\mathfontb=msbm10 at 11pt
\def\Bbb#1{\mbox{\mathfontb #1}}
\def\lesssim{\mbox{\mathfonta.}}
\def\suppset{\mbox{\mathfonta{c}}}
\def\subbset{\mbox{\mathfonta{b}}}
\def\grtsim{\mbox{\mathfonta\&}}
\def\gtrsim{\mbox{\mathfonta\&}}

\newcommand{\ar}{{\bf area}}
\newcommand{\1}{{\bf 1}}
\newcommand{\Bo}{\Box^{n}_{i}}
\newcommand{\Di}{{\cal D}}
\newcommand{\gd}{{\underline \gamma}}
\newcommand{\gu}{{\underline g }}
\newcommand{\ce}{\mbox{III}}
\newcommand{\be}{\mbox{II}}
\newcommand{\F}{\cal{F_{\eta}}}
\newcommand{\Ci}{\bf{C}}
\newcommand{\ai}{\mbox{I}}
\newcommand{\dupap}{\partial^{+}}
\newcommand{\dm}{\partial^{-}}
\newenvironment{note}{\begin{sc}{\bf Note}}{\end{sc}}
\newenvironment{notes}{\begin{sc}{\bf Notes}\ \par\begin{enumerate}}%
{\end{enumerate}\end{sc}}
\newenvironment{sol}
{{\bf Solution:}\newline}{\begin{flushright}
{\bf QED}\end{flushright}}

\title{
Rigidity and
non local connectivity of Julia sets of some
quadratic polynomials
}
\date{November 3, 2010}

\author{Genadi Levin
\\
\small{Inst. of Math., Hebrew University, Jerusalem 91904, Israel}\\
}
\normalsize
\maketitle
\abstract{For an infinitely renormalizable quadratic map
$f_c: z\mapsto z^2+c$ with the sequence of renormalization periods 
$\{k_m\}$ and rotation numbers $\{t_m=p_m/q_m\}$, we prove that 
if $\limsup k_m^{-1}\log |p_m|>0$, then the Mandelbrot set
is locally connected at $c$. We prove also that if 
$\limsup |t_{m+1}|^{1/q_m}<1$ and $q_m\to \infty$,
then the Julia set of $f_c$ is not locally connected
and the Mandelbrot set is locally connected at $c$
provided that all the renormalizations are non-primitive (satellite). 
This quantifies a construction
of A. Douady and J. Hubbard, and weakens a condition
proposed by J. Milnor.
}

\

\section{Introduction}\label{s1}

\begin{theo}\label{intromlc}
Suppose that for some quadratic polynomial
$f(z)=z^2+c_0$ there is an increasing sequence $n_m\to \infty$
of integers, such that
$f^{n_m}$ is simply renormalizable, and
\begin{equation}\label{m}
\limsup_{m\to \infty}\frac{\log |p_m|}{n_m}>0,
\end{equation}
where $p_m/q_m\in (-1/2, 1/2]$ denotes the rotation number 
(written in lowest terms
and with $q_m\ge 1$)
of the separating fixed point of the renormalization $f^{n_m}$.
Then $c_0$ lies in the boundary
of the Mandelbrot set $M$ and $M$ is locally connected at $c_0$.  

If, in addition, the renormalizations $f^{n_m}$ are non-primitive
(satellite), then the same conclusion holds under a weaker condition:
\begin{equation}\label{m1}
\limsup_{m\to \infty}\frac{\log q_m}{n_m}>0.
\end{equation}
\end{theo}
The sequence $\{n_m\}$ above is a subsequence
of the sequence of {\it all} renormalization periods of $f$.
If we suppose that $\{n_m\}$ represents  
all the renormalization periods of $f$, and all of them
are satellite and, moreover, $p_m/q_m$ are close enough to zero, then 
a weaker combinatorial condition is sufficient,
see Theorem~\ref{intronlc}(1)-(1') as well as Theorem~\ref{nlccoro}(1).

The following problems are central in
holomorphic dynamics:
MLC conjecture: ``The Mandelbrot set is 
locally connected'', and its dynamical counterpart:
``For which $c$ is the Julia set $J_c$ 
of $f_c$ locally connected?''. The MLC conjecture is equivalent to the
following rigidity conjecture:
if two quadratic polynomials with connected Julia sets 
and all periodic points repelling
are combinatorially equivalent,
then they are affinely conjugate. 
The MLC implies
that hyperbolic dynamics is dense in the space of complex quadratic
polynomials~\cite{DH2} (see also~\cite{shl}).
Yoccoz (see~\cite{H}) solved the above problems
for at most finitely renormalizable
quadratic polynomials as follows: for such a non hyperbolic map $f_c$,
the Julia set $J_c$ of $f_c$
is locally connected (provided
$f_c$ has no neutral periodic points), and
the Mandelbrot set $M$
is locally connected at $c$. 
(For the (sub)hyperbolic maps and maps with a parabolic point the
problem about the local connectivity of the Julia set
had been settled before in the works by Fatou, Douady and Hubbard
and others, see e.g.~\cite{CG}.) 
At the same time, infinitely renormalizable dynamical systems
have been studied intensively, see e.g.
~\cite{mcm},~\cite{mcm1},~\cite{gs} and references therein.
In his work on renormalization conjectures,
Sullivan~\cite{Su} (see also~\cite{ms})
introduced and proved so-called ``complex bounds''
for real infinitely renormalizable maps with bounded combinatorics. 
Roughly speaking, complex bounds mean that a sequence of renormalizations
is precompact. This property became in the focus of research.
It has played a basic role in recent breakthroughs in the problems of
local connectivity of the Julia set and rigidity for real and 
complex polynomials,
see~\cite{gsw},~\cite{ly},
~\cite{LS},~\cite{Shen},~\cite{KSS},\cite{KSS1},
~\cite{KL},~\cite{ABCD},~\cite{KS}.

On the other hand, there are maps without complex bounds. Indeed,
Douady and Hubbard showed the existence
of an infinitely renormalizable $f_c$, such that
its small Julia sets do not shrink
(and thus $f_c$ has no complex bounds) but still such that
the Mandelbrot set is locally connected at this $c$, see
~\cite{mil0},~\cite{Sor}.

Theorem~\ref{intromlc} provides a first 
class of combinatorics of infinitely renormalizable maps
$f_c$ without (in general) complex bounds,
for which the Mandelbrot set is locally connected at $c$.
Obviously, previous methods in proving the local connectivity of $M$
do not work in this case.
Our method is based
on an extension result for the multiplier
of a periodic orbit beyond the domain where
it is attracting, see next Section.

That the maps with the combinatorics described in
Theorem~\ref{intromlc} in general
do not have complex bounds and can have non locally connected Julia sets
follow from the second result, 
Theorem~\ref{nlccoro} stated below,
which makes the qualitative construction of Douady and Hubbard of an
infinitely renormalizable quadratic map with a non-locally connected 
Julia set, quantitative. 
It is based on the phenomenon known
as cascade of successive bifurcations (see e.g.~\cite{lold}), 
and is the following.  
Let $W$ be a hyperbolic component of the interior of the Mandelbrot set,
so that $f_c$ has an attracting periodic orbit of period $n$ 
for $c\in W$.
Given a sequence of rational numbers 
$t_m=p_m/q_m\not=0$ in $(-1/2, 1/2]$ 
(here and below $p_m, q_m$ are assumed to be co-primes
and with $q_m\ge 1$),
choose a sequence of hyperbolic components $W^m$ as follows:
$W^0=W$, and, for $m\ge 0$, 
the closure of the hyperbolic component $W^{m+1}$ touches the 
closure of the hyperbolic
component $W^{m}$ at the point 
$c_{m}$ with the internal argument $t_{m}=p_m/q_m$ (see next Sect.). 
When the parameter $c$ crosses $c_m$ from $W^m$ to $W^{m+1}$,
the periodic orbit of period $n_m=nq_0...q_{m-1}$ which is
attracting for $c\in W^m$ ``gives rise'' another periodic orbit
of period $n_{m+1}=n_m q_m$ which becomes attracting 
for $c\in W^{m+1}$. Thus when the parameter $c$ moves
to a limit parameter through the hyperbolic components
$W^m$, the dynamics undergoes a sequence of bifurcations 
precisely at the parameters $c_m$, $m=0,1,...$.

In the case when $n=1$ and $t_m=1/2$ for all $m\ge 0$,
the parameters $c_m$ are real, and we get the famous 
period-doubling cascade on the real line
known since 1960's~\cite{Myr}.
The corresponding limit parameter 
$\lim c_m=c_F=-1.4...$. The Julia set $J_{c_F}$ is
locally connected~\cite{HJ},~\cite{LS}. 

Douady-Hubbard's construction shows that if the sequence $\{t_m\}$
tends to zero fast enough, then 
the periodic orbits generated by the cascade of successive bifurcations
at the parameters $c_m$, $m>0$, 
stay away from the origin. This implies that the Julia sets
of the renormalizations of $f_{c_*}$ for $c_*=\lim c_m$ do not shrink,
and $J_{c_*}$ is not locally connected at zero. 
Their construction
is by continuity, and it does not give any particular sequence 
of $t_m$. By the Yoccoz bound for limbs, see~\cite{H}, 
$\{t_m\}$ can be chosen inductively so close to zero that
$M$ is locally connected at $c_*$.

Milnor suggests in~\cite{Mil}, p. 21,
that the convergence of the series:
\begin{equation}\label{mil}
\sum_{m=1}^\infty |t_m|^{1/q_{m-1}}<\infty
\end{equation}
could be a criterion that the periodic
orbits generated by the cascade 
after all bifurcations 
stay away from the origin.
Theorem~\ref{nlccoro} 
shows that the condition ~(\ref{mil}) is indeed 
sufficient for this, but not the optimal one (cf.~\cite{Lij}).
We give a weaker sufficient condition and thus prove:
\begin{theo}\label{nlccoro}
Let $t_0$, $t_1$,...,$t_m$,... be a sequence
of rational numbers $t_m=p_{m}/q_{m}\in (-1/2, 1/2]\setminus \{0\}$,
such that $\lim_{m\to \infty} q_m=\infty$ and
\begin{equation}\label{exp0}
\limsup_{m\to \infty} |t_m|^{1/q_{m-1}}<1.
\end{equation}
Let $W$ be a hyperbolic component of period $n$, and
the sequence of hyperbolic components $W^m$, $m=0,1,2,...$,
is built as above, so that $W^{m+1}$ touches $W^m$ at the point
$c_m$ with the internal argument $t_m$.
Then: 

(1) the sequence $\{c_m\}_{m=0}^\infty$
converges to a limit parameter $c_\ast\in \partial M$ and 
$M$ is locally connected at $c_\ast$, 

(2) the map $f_{c_\ast}$ is infinitely renormalizable
with non locally connected Julia set.
\end{theo}

If, for instance, $|t_{m+1}|^{1/q_m}\le 1/2$ for big indexes $m$, then
$q_{m+1}\ge 2^{q_m}\to \infty$, and Theorem~\ref{nlccoro} applies.

\

Let us make some further comments.

As it is mentioned above, we prove (2) by showing that
the conditions of Theorem~\ref{nlccoro} imply that
the cascade of bifurcated periodic orbits of $f_{c_*}$ stay away from
the origin.


The conclusion (2) of Theorem~\ref{nlccoro} breaks down 
if not all its conditions are valid. Indeed, 
if $t_m=1/2$ for all $m$, then $|t_{m+1}|^{1/q_m}=1/2^{1/2}<1$.
At the same time, as it is mentioned above, the limit Julia set $J_{c_F}$
is locally connected.

Theorem~\ref{nlccoro} is a consequence 
of a more general Theorem~\ref{intronlc}, see Section~\ref{js}.

Note in conclusion that there is a similarity 
between Theorems~\ref{nlccoro},~\ref{intronlc} 
and the celebrated
Bruno - Yoccoz criterion of the (non-)linearizability of quadratic
map near its irrational fixed point
\footnote[1]{This reflects a similarity between the phenomena of
non local connectivity of the Julia set in both cases. Recently Xavier Buff 
and Alexandre Dezotti proposed a far reaching
conjecture about this analogy~\cite{Buff}.}.

\

In Sect.~\ref{mult} we state Theorems~\ref{univ},~\ref{limb}
proved in~\cite{Lij}, and 
Theorems~\ref{univ1},~\ref{limb1} that we prove here, see last
Sect.~\ref{nprim}. 
Theorem~\ref{intromlc} is derived from Theorems~\ref{limb},~\ref{limb1}
in Sect.~\ref{rigidity}. In turn, Theorem~\ref{intronlc} is stated
and proved in Sect.~\ref{js}. It implies Theorem~\ref{nlccoro}.

\

Throughout the paper, $B(a,r)=\{z: |z-a|<r\}$.
We use both symbols $\exp(z)$ and $e^z$ to denote the exponential of 
$z\in {\bf C}$.

\

{\bf Acknowledgments.} The author is indebted to Alex Eremenko
for discussions, and especially for answering author's question about
the function $H$ and for finding
the reference~\cite{Ne} (see Section~\ref{js}).
The author thanks the referee for a careful reading the paper and
many helpful comments leading to several 
important improvements and corrections.
\section{Multipliers}\label{mult}
\subsection{Hyperbolic components}\label{hyper}
A component $W$ of the interior of $M$
is called an $n$-hyperbolic if 
$f_c$, $c\in W$, has an attracting periodic orbit $O_W(c)$
of period $n$.
Denote by $\rho_W(c)$ the 
multiplier of $O_W(c)$.
By the Douady-Hubbard-Sullivan theorem~\cite{DH1},~\cite{MS},~\cite{CG},
$\rho_W$
is a analytic isomorphism of $W$ onto the unit disk,
and it extends homeomorphically to the boundaries.
Given a number $t\in (-1/2, 1/2]$, denote by
$c(W, t)$ the unique point in $\partial W$ with the
{\it internal argument} $t$, i.e.
$\rho_W$ at this point is equal to $\exp(2\pi t)$. 
{\it The root} of $W$ is the point $c_W=c(W, 0)$ 
with the internal argument zero.

If $t=p/q$
is a rational number, we will always assume that $p, q$ are 
co-primes and $q\ge 1$.
For any rational $t\not=0$, denote
by $L(W, t)$ the connected component 
of $M\setminus \{c(W, t)\}$ which is disjoint with $W$. 
It is called the $t$-limb of $W$.
Denote also by $W(t)$ a $nq$-hyperbolic component
with the root point $c(W, t)$; its closure
touches $W$ at this point. The limb $L(W, t)$ contains $W(t)$.
The hyperbolic component $W$ is called {\it primitive}
if its root $c_W$ is not a point in the closure of other hyperbolic component.
Otherwise it is {\it non-primitive}.

Throughout the paper, we use a well-known notion of the external ray
and its angle (or argument), see e.g.
~\cite{CG},~\cite{DH1}~\cite{DH2},~\cite{levprz}.
External rays of the Julia set and the Mandelbrot set
will be called {\it dynamical rays} and {\it parameter rays}, respectively.

Let $W$ be an $n$-hyperbolic component. 
If $n>1$, the root $c_W$ of $W$ is the landing point of two
parameter rays with rational angles 
$0<\ell_{-}(W)<\ell_{+}(W)<1$.
If $n=1$, i.e. $W$ is the main cardioid, then $c_W=1/4$ 
is the landing point of
the only parameter ray of angle zero
(and one puts $\ell_{-}(W)=0, \ell_{+}(W)=1$ in this case). 
Period of each angle $\ell_\pm(W)$ under the doubling map
$\sigma: t\mapsto 2t(mod 1)$ is equal to $n$.
The angles $\ell_{\pm}(W)$ can be 
characterized as follows. The map $f=f_{c_W}$ has
a parabolic orbit $P$, and $\ell_\pm(W)$ are the nearest arguments
among arguments of the external rays of $f$, that land
at each point of $P$. See also Section~\ref{simple}.

The {\it wake} (see~\cite{H}) of a hyperbolic component $W$
is the only (open) component $W^*$ of the plane contaning $W$
cut by the parameter rays 
of arguments $\ell_\pm(W)$ 
(together with their common landing point $c_W$).
The points of the periodic orbit
$O_W(c)$ as well as its multiplier
$\rho_W$ extend as analytic functions to the wake $W^*$~\cite{H},~\cite{Leyo}.
Moreover, $|\rho_W|>1$ in $W^*\setminus \overline W$.

Finally, the following relation will be used. 
If $W$ is an $n$-hyperbolic component
and $p/q\not=0$ is rational, then
\begin{equation}\label{width}
\ell_{+}(W(p/q))-\ell_{-}(W(p/q))=\frac{(s_{+}-s_{-})(2^{n}-1)}{2^{n q}-1}.
\end{equation}
Here 
the integers $0\le s_{-} < s_{+}\le 2^{n}-1$ are the ``periods''
in the $2^{n}$-expansions of the angles $\ell_\pm(W)$ of the
root of $W$, i.e. $\ell_{\pm}(W)=s_{\pm}/(2^{n}-1)$.
This follows from Douady's tuning algorithm \cite{Douady}, 
see also~\cite{Leyo} 
(formula (6.1) with $d=2^n$ combined with Theorem 7.1) 
or Proposition 2.4.3 of~\cite{sch2}.

\subsection{Analytic extension of the multiplier}
Given $C>1$, consider an open set $\Omega$ of points 
in the punctured $\rho$-plane defined by the inequality
\begin{equation}\label{domeq}
|\rho-1| > C\log|\rho|
\end{equation}
It obviously contains the set $D_*=\{\rho: 0<|\rho|\le 1, \rho\not=1\}$
and is disjoint with an interval $1<\rho<1+\epsilon$.
Denote by $\Omega(C)$ the connected component of $\Omega$
which contains the set $D_*$ completed by $0$. 
Denote also by $\Omega^{\log}(C)$ the set of points 
$L=\log \rho=x+iy, \ \ \rho\in \Omega(C), \ \ |y|\le \pi$.
$\Omega(C)$ is simply-connected.
More precisely, the intersection of $\Omega^{\log}(C)$ with
any vertical line with $x=x_0>0$  is either empty
or equal to two (mirror symmetric) intervals.
If $C>4$, then $x< 2/(C-2)$ for all $L=x+iy$ in $\Omega^{\log}(C)$.
If $C$ is large enough, 
$\Omega^{\log}(C)$ contains two (mirror symmetric) domains
bounded by the lines $y=\pm (C/2)x$ ($x>0$) and $y=\pm \pi$

Let $W$ be an $n$-hyperbolic component of $M$.
The map $\rho_W$ from $W$
onto the unit disk
$c\mapsto \rho_W(c)$ has an inverse, which we denote by
$c=\psi_W (\rho)$. It is defined so far in the unit disk.
In~\cite{Lij} we prove the following.
\begin{theo}\label{univ}
 
(a) There exists $B_0>0$ as follows.
Suppose that, for some $c$, the map $f_c$ has
a repelling periodic orbit
of exact period $n$, and the multiplier 
of this orbit is equal to $\rho(c)$. Assume that $|\rho(c)|<e$. Then
\begin{equation}\label{main}
|\rho(c)-1|\le B_0\frac{4^n}{n} 
\{\log |\rho(c)|+\frac{|\rho'(c)|}{|\rho(c)|}(1+o(1))\}
\end{equation}
as $n\to \infty$.

(b) Denote
$\Omega_n=\Omega(n^{-1} 4^n B_0)$
and consider its log-projection
$$\Omega_n^{\log}=\Omega^{\log}(n^{-1} 4^n B_0)=\{L=x+iy: 
\exp(L)\in \Omega_n, |y|\le \pi\}.$$ 
Then the function $\psi=\psi_W$ extends to a holomorphic function
in the domain $\Omega_n$.

(c) The function $\psi$ is univalent
in a subset $\tilde\Omega_n$ of $\Omega_n$
defined by its log-projection
$\tilde\Omega_n^{\log}=\{\log\rho=x+iy: \rho\in \tilde\Omega_n, |y|\le \pi\}$ 
as follows:
$\tilde\Omega_n^{\log}=\Omega_n^{\log}\setminus
\{L: |L-R_n|<R_n\},$
where $R_n$ depends on $n$ only and has an asymptotics
$R_n=(2+O(2^{-n}))n\log 2$
as $n\to \infty$. 
Finally, the image of $\tilde\Omega_n$ by $\psi$ is contained
in the wake $W^*$.
\end{theo}
In the present paper we find a bigger extension for the function $\psi_W$
in the case of non-primitive $W$:
\begin{theo}\label{univ1}
There exists $\tilde K>0$ as follows. 
Let $W$ be a non-primitive $n$-hyperbolic component, i.e. $W=Z(t_0)$,
for some $n_0$-hyperbolic component $Z$ and some $t_0=p_0/q_0\not=0$,
and $n=n_0 q_0$.
Then $\psi_W$ extends to a univalent function in a domain 
$\Omega_{n_0,t_0}$ which consists of $\tilde\Omega_n$ and a neighborhood
of the point $1\in \partial \Omega_n$, and is defined by
its log-projection as follows:
$$\Omega_{n_0,t_0}^{\log}:=\{L=x+iy: \exp(L)\in \Omega_{n_0,t_0}, |y|\le \pi\}
=\tilde\Omega_n^{\log}\cup B(0, d),$$ 
where  $d=\tilde K \min\{ n |p_0| 4^{-n_0}, n^{-1}\}$.
If $q_0$ is large enough,
$\Omega_{n_0, t_0}^{\log}$ coincides with $\Omega_n^{\log}\cup B(0, d)$.
\end{theo}
Along with Theorem~\ref{univ}, the proof of Theorem~\ref{univ1} 
is based on geometric relations between the multiplier maps
$\rho_Z$ and $\rho_W$ of the hyperbolic components $Z$ and $W=Z(t_0)$
respectively near the common point $c_0=c(Z, t_0)=c(W, 0)$ of their 
boundaries. We start with the 
following known relation~\cite{Gu}: 
$|(d\rho_W/d\rho_Z)(c_0)|=q_0^2$.
Then we show in Lemmas~\ref{q} and~\ref{r}-\ref{rb} that
$\rho_W$ {\it as a function of $\rho_Z$} extends to a {\it univalent} 
function in a disk centered at the point $\rho_Z(c_0)$ and
of radius proportional to
$r(n_0, p_0/q_0)=\min\{1/(n_0 q_0^3), 2^{-n_0/2}|p_0/q_0|\}$ 
as far as the function $\psi_Z=\rho_Z^{-1}$
allows to do such an extension, i.e. univalent. The latter
holds in the domain $\tilde\Omega_{n_0}$.
It follows that $d$ is roughly 
$q_0^2 \min\{r(n_0, p_0/q_0), dist(2\pi i p_0/q_0,
\partial \tilde\Omega^{\log}_{n_0})\}$.
See Sect.~\ref{univ1pr} for the complete proof.
Note that the proof of Lemma~\ref{q} uses some classical
tools as well as bounds for multipliers from~\cite{pet1}
and~\cite{levcoll}, see Sect.~\ref{bif} and Appendix. 
This Lemma provides also a step
in proving Theorem~\ref{intronlc}.
\subsection{Limbs}\label{limbs}
Let $W$ be an $n$-hyperbolic component.
For every $t=p/q\not=0$, consider the hyperbolic component
$W(t)$, the corresponding
wake $W(t)^*$ and the limb $L(W,t)$.
Then a branch $\log\rho_W$
is well-defined in the open set $W(t)^*$, such that 
$\log\rho_W(c)\to 2\pi i t$ 
as $c\to c(W,t)$, and, for every $c\in L(W,t)$,
the point $\log\rho_W(c)$ is 
contained in the following round disk (Yoccoz's circle):
\begin{equation}\label{yoccircle}
Y_n(t)=\{L: |L-(2\pi it+\frac{n\log 2}{q})|<\frac{n\log 2}{q}\},
\end{equation}
see~\cite{H},~\cite{Pe} and references therein. 
See also Theorem~\ref{yoccircledis} of Appendix.
\begin{com}\label{lcw}
An important corollary of the Yoccoz bound~(\ref{yoccircle}) is that
every point in $M\cap \overline{W^*}$ either belongs to the closure
of the hyperbolic component $W$, or belongs to some its limb $L(W, p/q)$
(see~\cite{H},~\cite{Mi2}).
\end{com}
Theorem~\ref{intromlc} will be a consequence of the following
two bounds on the size of limbs. The first one is proved in~\cite{Lij} 
and the proof is based on Theorem~\ref{univ}(a)-(b) and~(\ref{yoccircle}).
\begin{theo}\label{limb}
There exists $A>0$, such that, for 
every $n$-hyperbolic component $W$ and every $t=p/q\in (-1/2, 1/2]$,
the diameter of the limb $L(W, t)$ is bounded by:
\begin{equation}\label{allbound}
diam L(W, t)\le A \frac{4^n}{|p|}.
\end{equation}
\end{theo}
The second bound concerns the non-primitive components.
Its proof is based on 
Theorems~\ref{univ}(a),~\ref{univ1},~\ref{limb}, and on~(\ref{yoccircle}), 
and it states the following.
\begin{theo}\label{limb1}
There exists $\tilde A>0$, such that, 
if an $n$-hyperbolic component $W$ is not primitive,
then
\begin{equation}\label{allbound1}
diam L(W, p/q)\le \tilde A \frac{8^n}{q}.
\end{equation}
\end{theo}
See Sect.~\ref{limb1pr} for the proof.

\section{Rigidity}\label{rigidity}
\subsection{Simple renormalization}\label{simple}
We follow some terminology as in~\cite{mcm}, see also~\cite{Mi2}.
For the theory of polynomial-like maps, see~\cite{DH3}.
Let $f$ be a quadratic polynomial with connected Julia set.
The map $f^n$ is called {\it renormalizable} if there are open
disks $U$ and $V$ such that $f^n: U\to V$ is a polynomial-like
map with a single critical point at $0$ and with connected
Julia set $J_n$. The map $f^n: J_n\to J_n$ has two fixed points
counted with multiplicity:
$\beta$ (non-separating) and $\alpha$.
Denote them by $\beta_n$ and $\alpha_n$.
The renormalization is {\it simple}
if any two small Julia sets $f^i(J_n)$, $i=0,1,...,n-1$
cannot cross each other,
i.e. they can meet only at some iterate of $\beta_n$.
If these small Julia sets don't meet, the renormalization
$f^{m_n}$ is called of {\it disjoint type}, or {\it primitive}.
Otherwise it is called of {\it $\beta$-type}, {\it non-primitive},
or {\it satellite type}.

Every repelling periodic point $z$ of $f$ of period $n$ has
a well-defined rational {\it rotation number} $p/q\in (-1/2, 1/2]$,
which is defined by the order at which 
$f^n$ permutes (locally) $q$ external rays landing at $z$ in the
couterclockwise direction.
If the fixed point $\alpha_n$ of $f^n$ is repelling, then it has
a non-zero rational rotation number $p/q$,
which can be defined equivalently as follows:
$f^n: U\to V$ is hybrid equivalent to a quadratic  polynomial
which lies in the $p/q$-limb of the main cardioid.

\subsection{Demonstration of Theorem~\ref{intromlc}}
We split the proof into few steps.

{\it A.} 
Let $f^n: U\to V$ be a simple renormalization of $f$,
$\beta_n$ its $\beta$-fixed point, and
$O_n=\{f^i(\beta_n)\}_{i=0}^{n-1}$ the periodic orbit contaning $\beta_n$.
We use some notions and results from~\cite{DH2},~\cite{Mi2}.
The characteristic arc $I(O_n)=(\tau_-(O_n), \tau_+(O_n))$ of $O_n$
is the shortest arc (measured in $S^1$)
between the external arguments of the
rays landing at the points of $O_n$. Then 
$\tau_\pm(O_n)$ are the arguments of 
two dynamical rays that land at the point 
$\beta_n'=f(\beta_n)$ of $O_n$, and $c_0=f(0)$ lies in the sector
bounded by these rays and disjoint with $0$.
Furthermore, the two {\it parameter rays} of the same arguments 
$\tau_\pm(O_n)$ land
at a single parameter $c(O_n)$. The point $c(O_n)$ is the root 
of a hyperbolic component denoted by $W_n$, and
the above parameter rays  completed by  $c(O_n)$ bound
the wake $W_n^*$ of this component. Denote by $L(O_n)=W_n^*\cap M$
the corresponding limb.
Thus,

{\it (A1)} $c_0\in L(O_n)$,

{\it (A2)} moreover, the rotation
number of the $\alpha$-fixed point $\alpha_n$ of the renormalization
is $p/q$ if and only if $c_0\in L(W_n, p/q)$ - the limb of $W_n$
which is attached at the point of $\partial W_n$
with the internal argument $p/q$.

{\it B.} Let $m<m'$. By~\cite{mcm}, since all renormalizations $f^{n_m}$
are simple, $n_m$ divides $n_{m'}$.
Furthermore, by (A2),

{\it (B1).}
$$\tau_{-}(O_{n_m})<\ell_-(W_{n_m}(p_m/q_m))\le\tau_-(O_{n_{m'}})<\tau_+(O_{n_{m'}})\le
\tau_+(W_{n_m}(p_m/q_m))<\tau_{+}(O_{n_m}),$$ 
thus the sets
$\overline L(O_{n_m})$ form a decreasing sequence of compact sets
containing $c_0$. 
Therefore, for the local connectivity of $M$ at $c_0$ 
it is enough to prove that the set
$$S=\cap_{m}\overline L(O_{n_m})$$
consists of a single point.

{\it (B2).} It is easy to see that $\tau_+(O_{n_m})-\tau_-(O_{n_m})\to 0$.
Indeed, by (B1) and~(\ref{width}), 
$\tau_+(O_{n_m})-\tau_-(O_{n_m})\le(2^{n_m}-1)^2/(2^{n_m q_m}-1)$, and $q_m>2$.
Thus, there exists
\begin{equation}\label{exlim}
\lim \tau_\pm(O_{n_m})=\tau_0.
\end{equation}
Note that $\tau_0$ is not periodic under the doubling
map 
$\sigma(t)=2t(mod \ 1)$.

{\it C.} Let $c$ be any point from $S$.

{\it (C1)} All periodic points of $f_c$ are 
repelling. Indeed, obviously, $f_c$ cannot have an attracting
cycle. If $f_c$ has an irrational neutral
periodic orbit then $c$ lies in the boundary of a hyperbolic
component contained in $S$, a contradiction with ~(\ref{exlim}).
If $f_c$ has a 
neutral parabolic periodic orbit, then $c$ is the landing
point of precisely two parameter rays with periodic arguments,
if $c\not=1/4$, and 
the only ray landing at it, of zero argument, if $c=1/4$,
again the same contradiction.
Thus, all cycles are repelling.
Consider the so-called real lamination $\lambda(c)$
of $f_c$~\cite{Ki0}. It is a minimal closed equivalence relation
on $S^1$ that identifies two points whenever their prime end impressions
intersect.
For every $m$, $\{\tau_{-}(O_{n_m}), \tau_{+}(O_{n_m})\}\subset \lambda(c)$.
Since $\lambda(c)$ is closed, $\tau_0\in \lambda(c)$. Moreover, for every
$m$, there is a pair of angles $\{\tau^-_m, \tau^+_m\}$, such that
their images under the doubling map
$\{\sigma(\tau^\pm_m)\}\subset \{\tau_\pm(O_{n_m})\}$,
and $\{\tau^-_m, \tau^+_m\}$ is contained
in the class of $\lambda(c)$ corresponding to the point $\beta_{n_m}$.
Passing to the limit, we get that
the pair $\{\tau_0/2, \tau_0/2+1/2\}$ is the critical class
of $\lambda(c)$.
The following statements are known after Thurston
and Douady and Hubbard and proved in a much more
general form in~\cite{Ki},
Proposition 4.10: if the critical class 
$\{\tau_0/2, \tau_0/2+1/2\}$ is contained
in a class of $\lambda(c)$, then $\lambda(c)$ is determined
by (the itinerary of) $\tau_0$. Since $\tau_0$ is the same for all $c\in S$,
we conclude that the real laminations of all $f_c$, $c\in S$,
coincide. In particular, for every $c\in S$,
$f_c^{n_m}$ is simply renormalizable because by~\cite{mcm}
the simple renormalizations are detected by the lamination.  
The renormalization $f_c^{n_m}$ is hybrid equivalent to
some $f_{T(c,m)}$~\cite{DH3}. Denote $\hat n_k=n_{m+k}/n_m$, $k>0$. Then
$f_{T(c,m)}^{\hat n_k}$ is simply renormalizable, and the rotation number
of its $\alpha$-fixed point is $p_{m+k}/q_{m+k}$ because it is
a topological invariant.

{\it (C2)} By (A2), $T(c,m)$ is contained
in the intersection of a decreasing sequence of limbs $L_{m,k}$,
$k=1,2,...$, such that $L_{m,k}$ is
attached to a hyperbolic component
of period $\hat n_k$ at the internal argument $p_{m+k}/q_{m+k}$.
By Theorem~\ref{limb}, the diameter of $L_{m, k}$,
$$diam L_{m,k}\le A \frac{4^{\hat n_k}}{|p_{m+k}|}= 
A \frac{4^{\frac{n_{m+k}}{n_m}}}{|p_{m+k}|}.$$
Since $\limsup  (\log |p_m|)/n_m>0$, one can find and fix
$m$ in such a way, that 
$$\liminf_{k\to \infty}\frac{4^{\frac{n_{m+k}}{n_m}}}{|p_{m+k}|}=0.$$
It means, that, for the chosen $m$, the limbs
$L_{m,k}$ ($k\to \infty$) shrink to a point $\hat c=T(c,m)$, 
so that $\hat c$ depends on $m$ but is independent of $c\in S$.
Thus 
$f_c^{n_m}$ is quasi-conformally conjugate
to $f_{\hat c}$, for all $c\in S$.

{\it D.} 
Assume that the compact $S$ has at least two different points.
Then $S$ contains a point $c_1\in \partial M$ different from $c_0$.
By (C1)-(C2), the renormalizations $f_{c_0}^{n_m}$
of $f_{c_0}$ and $f_{c_1}^{n_m}$ of $f_{c_1}$
are quasi-conformally conjugate,
all periodic points of $f_{c_0}$, $f_{c_1}$ are repelling,
and $\lambda(c_0)=\lambda(c_1)$.
We are in a position to apply Sullivan's pullback argument (see~\cite{ms}).
Using a quasi-conformal conjugacy near small Julia sets (rather than
on the postcritical set) and an appropriate puzzle structure, 
we arrive at a quasi-conformal conjugacy
between $f_{c_0}$, $f_{c_1}$. Since $c_1\in \partial M$,
then according to~\cite{DH3}, $c_0=c_1$.   

This finishes the proof of Theorem~\ref{intromlc} under
the condition~(\ref{m}).

{\it E.} Now assume that every renormalization
$f^{m_n}$ is also non-primitive.
Then the hyperbolic components $W_{n_m}$ are non-primitive 
as well~\cite{Mi2},
and we can apply the bound of Theorem~\ref{limb1} instead
of Theorem~\ref{limb}. It gives the same conclusions
under the weaker condition~(\ref{m1}). The proof stands the same
with some obvious changes in (C2).

\section{Non locally connected Julia sets}\label{js}
\subsection{Statement}
Let $t_0$, $t_1$,...,$t_m$,... be a sequence
of non-zero rational numbers $t_m=p_{m}/q_{m}\in (-1/2, 1/2]$.
Let us introduce the following conditions.

$(Y0)_a$ 
\begin{equation}\label{Yrel0a}
\sup_{m\ge 1} |t_m|q_0...q_{m-1}<\infty.
\end{equation}

$(Y0)_b$
\begin{equation}\label{Yrel0b}
\inf_{m\ge 2} \frac{\log\frac{|p_{m-1}q_{m-1}|}{|t_m|}}{q_0...q_{m-2}}>0.
\end{equation}

$(Y1)$ There exists $\beta>0$, such that 
\begin{equation}\label{Yrel01}
\limsup_{m\to \infty} \frac{q_m}{\max\{(q_0...q_{m-1})^2, \
 \exp(\beta q_0...q_{m-2})\}}>0.
\end{equation}

$(S)$ for some $k\ge 0$ and $\gamma>0$,
\begin{equation}\label{Srel0}
\sum_{m=k}^\infty  \frac{u_{k,m}}{q_m(1-u_{k,m})}H(u_{k,m+1})<\infty, 
\end{equation}
where 
\begin{equation}\label{Rintro0}
u_{k,m}=|t_{m+1} 
\max\{q_k...q_m, \ \frac{\exp(\gamma q_k...q_{m-1})}{|p_mq_m|}\}|^{1/q_m}, 
m\ge k
\end{equation}
(we set $q_k...q_{m-1}=1$, if $m=k$),
and
\begin{equation}\label{h}
H(u)=16u \Pi_{k=1}^\infty \frac{(1+u^{2k})^8}{(1-u^{2k-1})^8}.
\end{equation}
In particular, $H: [0,1)\to [0, \infty)$ is strictly increasing 
from zero to infinity,
and extends to a holomorphic function in the unit disk.
For more information about $H$, see Subsection~\ref{H}.

See also a remark on the condition (S) in the beginning of 
Subsection~\ref{lcpr}.

\

Let now $W$ be a hyperbolic component of some period $n\ge 1$, 
and the sequence $\{W^m\}$ of hyperbolic component is built as 
in the introduction,
in other words, $W^0=W$, and the closure of the 
hyperbolic component $W^{m+1}$ touches the closure of the hyperbolic
component $W^{m}$ at the point $c_{m}\in \partial W^{m}$
with internal argument $t_{m}$. 
\begin{theo}\label{intronlc}
The following statements hold.

1. If the conditions $(Y0)_a$-$(Y0)_b$ are satisfied, then 
the sequence of parameters $c_m$, m=0,1,2,...,
converges to a limit parameter $c_\ast$.

1'. If, additionally, the condition $(Y1)$ is satisfied, then 
the Mandelbrot set is locally connected at the limit parameter $c_\ast$.

2.  If the conditions $(Y0)_a$-$(Y0)_b$ and $(S)$ are satisfied, then 
the map $f_{c_\ast}$ is infinitely renormalizable
with non locally connected Julia set.
\end{theo}

Theorem~\ref{nlccoro} stated in the Introduction
is a simple corollary of Theorem~\ref{intronlc}:
\begin{prop}\label{7imp2}
Assume that $lim_{m\to \infty}q_m=\infty$ and, for some $a<1$
and all $m$ large enough, 
\begin{equation}\label{ta}
|t_{m+1}|\le a^{q_{m}}. 
\end{equation}
Then the 
conditions $(Y0)_{a-b}$, $(Y1)$, and $(S)$ are satisfied.
\end{prop}
\begin{proof}
One can assume that, for every $m\ge m_0$,
$q_m$ is large enough and (\ref{ta}) holds.
Then (\ref{ta}) implies $q_{m+1}>q_{m}^3$, for $m\ge m_0$,
which, in turn, implies that
\begin{equation}\label{ozenka}
q_m>(q_{m_0}...q_{m-1})^2 q_{m_0},
\end{equation} 
for $m>m_0$. 
Using this, we get, for $m\ge m_0+2$, 
$$q_m/(1/a)^{q_0...q_{m-2}}\ge (1/a)^{q_{m-1}}a^{q_0...q_{m-2}}=
(1/a)^{q_{m-1}-q_0...q_{m-2}}\to \infty.$$ 
In particular, $(Y1)$ holds with $\beta=\log(1/a)$. 
As for $(Y0)_{a-b}$, we have (for $m\ge m_0+2$):
$$|t_m|q_0...q_{m-1}<a^{q_{m-1}}q_0...q_{m_0-1}(q_{m_0}...q_{m-2})q_{m-1}<
a^{q_{m-1}}q_{m-1}^{3/2}(q_0...q_{m_0-1}),$$ 
and the latter sequence is bounded
(it tends to zero). It proves $(Y0)_a$. In turn,
by~(\ref{ta})-(\ref{ozenka}),
$$\frac{\log|p_{m-1}q_{m-1}/t_m|}{q_0...q_{m-2}}
\ge 
\frac{q_{m-1}\log (1/a)}{q_0...q_{m-2}}\to \infty,$$
which proves $(Y0)_b$.
Let us varify $(S)$. Using~(\ref{ozenka}) with $k>m_0$ instead of $m_0$, 
we have, for $m\ge k$: 
$u_{k, m}<a 
\max\{(q_m^{3/(2q_m)}, \exp(\gamma q_k...q_{m-1}/q_m)\} <a_1$, 
for some $a<a_1<1$ provided $k\ge m_0$ is large enough and $m\ge k$. 
Fixing such $k$, we have for $m\ge k$,
$H(u_{k, m+1})u_{k, m}/(1-u_{k, m})<H(a_1) a_1/(1-a_1)$.
On the other hand, $q_m>2^{3^{m-k}}$, $m\ge k$.
Therefore, (\ref{Srel0}) holds, too.
\end{proof}
Note that, with the help of the bound~(\ref{bh}), see Subsection~\ref{H},
one can easily find sequences $t_m$, such that 
$|t_m|^{1/q_{m-1}}\to 1$ and such that
the conditions of Theorem~\ref{intronlc} hold. 

\

The rest of the Section occupies the proof of Theorem~\ref{intronlc}.
\subsection{Bifurcations}\label{bif}
Let $W$ be an $n$-hyperbolic component,
and let $c_0\in \partial W$ have an internal argument
$t_0=p/q\not=0$. 
Consider the periodic orbit $O(c)=\{b_j(c)\}_{j=1}^n$ 
of $f_c$ which is attracting when $c\in W$
(that is, $O$ is the orbit denoted by $O_W$ in Subsection~\ref{hyper}).
Then all $b_j(c)$ as well as the mulptiplier $\rho_W(c)$ of $O(c)$
are holomorphic in $W$ and
extend to holomorphic functions in $c$ in the whole wake $W^*$ of $W$.
As we know, the function $\rho_W(c)$ is injective
near $c_0$. Consider the inverse function
$\psi_W$. It is well defined and univalent 
in the domain $\tilde \Omega_n$,
which includes the unit disk 
and a neighborhood of the point
$$\rho_0=\exp(2\pi i p/q),$$ 
so that $\psi_W(\rho_0)=c_0$.
It is convenient to use also the composition
$$\psi_W^{\log}=\psi_W\circ \exp.$$ 
It is defined and holomorphic in the
domain $\tilde\Omega_n^{\log}$, which includes the left half-plane
$\{L: Re(L)<0\}$ and a neighborhood of the point $2\pi i p/q$.
Recall that $W(p/q)$ denotes a hyperbolic component
touching $W$ at the point $c_0$. 

The following well-known picture describes the (local) bifurcation near $c_0$.
For the proof, see e.g~\cite{Che} (for $n=1$), \cite{Lij} or \cite{Mi2}.
Let us fix a disk $B(0,\delta)$,
where $\delta>0$ is so small,
that $\psi_W$ is univalent in $B(0,1)\cup B(\rho_0, \delta^q)$.
Given $s\in B(0,\delta)$,
define $\rho=\rho_0+s^q$ and $c=\psi_W(\rho)$. Fix a small
neighborhood
$E$ of the set $O(c_0)$. 
For $1\le k\le n$, there exists a function $F_k$, which
is defined and holomorphic in $B(0,\delta)$,
$F_k(0)=0$, $F_k'(0)\not=0$, such that,
for every $s\in B(0,\delta)$, $s\not=0$,
the points 
$b_{k,j}^{p/q}(s)=b_k(c_0)+F_k(s\exp(2\pi i \frac{j}{q}))$, 
$1\le k\le n$, $0\le j\le q-1$,
are the only fixed points of $f_c^{nq}$ in the neighborhood $E$,
which are different from the points of $O(c)=\{b_k(c)\}_{k=1}^n$. 
They form a periodic orbit $O^{p/q}(c)$ of $f_c$ of period $nq$,
which collides with $O(c)$ as $c\to c_0$.
Denote $\hat B=\psi_W(B(\rho_0, \delta^q))$.
The multiplier of $O^{p/q}(c)$ is the product
$2^{nq}\Pi_{1\le k\le n, 0\le j\le q-1}b_{k,j}^{p/q}(s)$, 
which is invariant under the change $s\mapsto s\exp(2\pi i/q)$. 
Hence,
this multiplier is, in fact, 
a non-constant holomorphic function on $c\in \hat B$, which takes the value
$1$ at $c=c_0$. 
As the map $f_c$ has at most one non-repelling
periodic orbit, the cycle $O^{p/q}(c)$ is attracting for
$c\in \hat B\cap W(p/q)$. 
Therefore, for such $c$, $O^{p/q}(c)$ is just the cycle
$O_{W(p/q)}(c)$ of period $nq$, which exists and attracting
throughout $W(p/q)$. In particular,
the multiplier of $O^{p/q}(c)$ is just the multiplier $\rho_{W(p/q)}(c)$
of the attracting periodic orbit of $f_c$, for $c\in W(p/q)$.

Let us make a general remark.
Assume that, for some $m\ge 1$
and for any $c$ in some domain $\Omega$, the map 
$f_c^m$ has no fixed points with multiplier $1$.
(For example, this is the case, for any $m$, if $\Omega$ is
a hyperbolic component.)
Then, by the Implicit Function Theorem, 
every fixed point of $f_c^m$ as well as
its multiplier is defined locally as a holomorphic
function, which has an analytic contination along every curve in $\Omega$.
As for the continuation
of the multiplier function, a weaker condition is enough.
By the above local bifurcation picture, the multiplier of a periodic orbit
of $f_c$ of period $m$ extends analytically through a neighborhood
of any parameter $\hat c$ unless $f_{\hat c}$ 
has a periodic orbit of (exact) period $m$ with multiplier $1$
(i.e., $\hat c$ is the root of a primitive $m$-hyperbolic component).
Assume now that, for any $c\in\Omega$, the map 
$f_c$ has no periodic orbits of period $m$
with multiplier $1$. Then we have, that the multiplier
of any periodic orbit of $f_c$ of period $m$, which is defined locally
near $c\in \Omega$, has an analytic continuation along every curve 
in $\Omega$, which starts at $c$. 

We will be concerned with the problem of holomorphic ($=$analytic)
extensions ($=$continuations) of the multiplier
functions $\rho_W$ and $\rho_{W(p/q)}$ from a domain to a bigger domain.
As the multiplier $\rho_{W(p/q)}(c)$ is holomorphic in a small neighborhood
$\hat B$ of $c_0$ and, by the above, it extends from $\hat B$
to holomorphic functions
defined in $W(p/q)$ and in $W$,
$\rho_{W(p/q)}$ extends to a holomorphic function defined
in the simply-connected domain $\hat B\cup W\cup W(p/q)$.
Recall also that $\rho_{W(p/q)}$
has an analytic continuation from $W(p/q)$
to the wake $W(p/q)^*$, and $|\rho_{W(p/q)}|>1$ in
$W(p/q)^*\setminus \overline{W(p/q)}$. 
Thus $\rho_{W(p/q)}$
is holomorphic in the domain $W\cup \hat B\cup W(p/q)^*$,
and $|\rho_{W(p/q)}|>1$ in
$(W\cup W(p/q)^*)\setminus \overline{W(p/q)}$. 
 
Now, since $\rho_W$ is univalent in $W\cup \hat B$, the function
$\rho_W(p/q)$ is
an implicit function of $\rho=\rho_W$ in $B(0,1)\cup B(\rho_0, \delta^q)$. 
We study $\rho_{W(p/q)}$ as a function of $\rho_W$ whenever 
it makes sense.
The following relation between $\rho_{W(p/q)}$ and $\rho_W$
at the bifurcation parameter $c_0$
is proved in~\cite{Gu}:
\begin{equation}\label{gu} 
\frac{d\rho_{W(p/q)}}{d\rho_W}(c_0)=-\frac{q^2}{\rho_0}.
\end{equation}
Another important ingredient for us is an inequality connecting
the multipliers $\rho_{W(p/q)}(c)$ and $\rho_W(c)$
when $c$ lies in the hyperbolic component $W$.
For $c\in W$, such that $\rho_W(c)\not=0$, the following bound takes place:
\begin{equation}\label{inside}
\frac{|\log\rho_{W(p/q)}(c)|^2}{\log|\rho_{W(p/q)}(c)|}
<q^2\frac{|\log\rho_W(c)-2\pi i p/q|^2}
{-\log|\rho_W(c)|},
\end{equation}
for some branch of $\log\rho_{W(p/q)}(c)$ and any branch
of $\log\rho_W(c)$.
For the proof, see Appendix.

Let us introduce the function
$$\Psi_{W,p/q}=\rho_{W(p/q)}\circ \psi_W^{\log}.$$
It is holomorphic in the left half-plane union with a neighborhood
of the boundary point $2\pi i p/q$, and 
$\Psi_{W,p/q}(2\pi i p/q)=1, \ \ \Psi_{W,p/q}'(2\pi i p/q)=-q^2$.
(The latter holds by~(\ref{gu}).) We want to know how far $\Psi_{W,p/q}$ is
univalent.
The main technical part is contained in the next Lemma~\ref{q}
having an independent interest.

Suppose $g: B(0,1)\to {\bf C}$ is a univalent function,
and $U$ a simply-connected domain, such that
$B(0,1)\cup U$ is also a simply-connected domain.
(Below, $U$ will be either a disk
or the image of a disk
by exponential map.)
We say
that $g$ has a {\it univalent extension
to} $B(0,1)\cup U$, if there is a function
(denoted by the same letter $g$),
which is holomorphic in $B(0,1)\cup U$,
coincides with $g$ in $B(0,1)$ and is (globally) univalent in $B(0,1)\cup U$.
\begin{lem}\label{q}
For every $X>0$ there exist $0<\Lambda_0<\Lambda<1$ depending only on $X$,
such that the following properties hold. 
Assume that, for some $0<r<1$, 
the function $\psi_W$ has a 
univalent extension
to $B(0,1)\cup U$, where 
$$U=\exp(B(2\pi i p/q, r))=\{\exp(w): w\in B(2\pi i p/q, r)\}.$$
Assume further that 
the topological disk
$V=\psi_W^{\log}(B(2\pi i p/q, r))=\psi_W (U)$ 
containing $c_0$ 
obeys the following two disjointness properties:

(a) $V$ is disjoint with 
any limb $L(W, p'/q')$ of $W$ 
other than $L(W, p/q)$ and such that $q'\le q+1$,

(b) $V$ is disjoint with the subset of $M$, which 
is outside of the wake of $W$:
$$V\cap (M\setminus W^*)=\emptyset.$$
Besides, $V$ is disjoint with the parameter ray of argument zero,
i.e. with the set $R_0=\{c>1/4\}$.

Then the following conclusions hold.

I. The function
$\rho_{W(p/q)}(c)$ extends to a holomorphic function defined in 
the domain $V^{p/q}:=V\cup W(p/q)^*\cup W$, 
and, for each $\lambda\in \overline{B(0,1)}$ there is a unique
$c_\lambda\in V^{p/q}$, such that $\rho_{W(p/q)}(c_\lambda)=\lambda$.
Clearly, $c_\lambda\in \overline{W(p/q)}$.

II. 
If $r q^2<X$, then
the function $\Psi_{W,p/q}$ is well-defined and univalent in the disk
$B(2\pi i p/q, \Lambda r)$, and
$\frac{2}{3}q^2<|\Psi_{W,p/q}'(L)|<\frac{3}{2}q^2$,
for every $L\in B(2\pi i p/q, \Lambda_0 r)$.
In particular, the following covering property holds:
for every $\theta\le \Lambda_0 r$, the image of 
$B(2\pi i p/q, \theta)$ under the map $\Psi_{W, p/q}$
covers $B(1,2 q^2 \theta/3)$ and is covered by 
$B(1,3 q^2 \theta/2)$.

Furthermore, the function $\psi_{W(p/q)}$ 
(the inverse to $\rho_{W(p/q)}: W(p/q)\to B(0,1)$) has a 
univalent extension to $B(0,1)\cup B(1,2 q^2 \Lambda_0 r/3)$.
\end{lem}
\begin{com}\label{disj}
Roughly speaking, the conditions (a) and (b) guarantee that
the function $\rho_{W(p/q)}$ extends analytically through every point
in $V\cap W^*$ and $V\setminus W^*$ respectively.
The restriction (a) is ``local'' (inside of the wake) while
(b) is a ``global'' one (in the rest of $M$). 
We give bounds on $r$ from below in terms of $n$ and $p/q$ to satisfy
(a) and (b) in Lemma~\ref{r} and Lemma~\ref{rb} respectively. 
\end{com}
The following combinatorial fact appears in Lemma 6.1 of
\cite{Lij} (in a slightly different form). 
For completeness, we reproduce its short proof here:
\begin{prop}\label{combin}(cf.~\cite{Che})
Let $W$ be an $n$-hyperbolic component.
Let also $c\in L(W, t')\cup c(W, t')$, for some
$t'=p'/q'$ and $q'>2$. Assume that $f_c^{nQ}$ has a fixed point 
with the multiplier $1$. Then
$$Q\ge q'-1$$
\end{prop}
\begin{proof}
Consider the dynamical plane of $f_c$.
The critical value $c$ lies in the sector $S$
bounded by the dynamical rays of arguments $\ell_\pm(W(t'))$
and disjoint with $0$. On the 
other hand, $c$ belongs to a petal at a fixed point
$a$ of the map $f_c^{nQ}$ with the multiplier $1$.
Hence, $a$ is in the closure of the same sector $S$, too.
Since $(f_c^{nQ})'(a)=1$, every dynamical ray of $f_c$,
which lands at $a$, is fixed by
$f_c^{nQ}$ (see e.g.~\cite{Milbook}). But since $c\not=1/4$,
$a$ is a landing point of at least two rays. Therefore, there are two rays
$R_{t_1}, R_{t_2}$, $0<t_1<t_2<1$ (landing at $a$), which are fixed by
$f_c^{nQ}$ and which lie in the closure of $S$.  
Then
$\ell_{+}(W(t'))-\ell_{-}(W(t')) \ge t_2-t_1\ge (2^{nQ}-1)^{-1}$.
Apply the formula (\ref{width}). It gives us
$nQ\ge nq'-2n+1$, that is, $Q\ge q'-1$.   
\end{proof}
Now we start the proof of Lemma~\ref{q}.
We drop some indices and write
$$\psi=\psi_W, \ \ \rho_{p/q}=\rho_{W(p/q)},$$ 
and also
$$\Psi=\Psi_{W, p/q}, \ \ \ B=B(2\pi i p/q, r).$$
First, we show Part I. 
We prove a more general statement, see Proposition~\ref{rhocont} below.
Given the hyperbolic component $W$ and the point $t_0=p/q\not=0$,
we define a simply-connected domain $D(W, p/q)$ as follows:
\begin{equation}\label{D}
D(W, p/q)={\bf C}\setminus (R_0\cup M_1\cup M_2\cup M_3).
\end{equation}
Here $R_0=\{c>1/4\}$ is the parameter ray to the root $c=1/4$ of the
main cardioid, and $M_i$, $1\le i\le 3,$ are the following subsets
of $M$:

(i) $M_1=M\setminus W^*$. Note that $M_1$ is a continuum, which contains
$c_W$ and $c=1/4$ (unless $W$ is the main cardioid and $M_1=\{1/4\}$),

(ii) $M_2=\cup L(W, p'/q')$, over all $p'/q'\not=p/q$, such that $q'\le q+1$,

(iii) $M_3$ is the shortest subarc of the simple closed curve $\partial W$,
which contains all $c(W, p'/q')$ with $p'/q'$ as in (ii), i.e., 
$p'/q'\not=p/q$ and $q'\le q+1$. In other words, if,
for $0<t_1<t_2<1$, we denote $l(t_1, t_2)=\{c(W, t): t_1<t<t_2\}$
(the open subarc of $\partial W$ with the end points $c(W, t_1)$
and $c(W, t_2)$, which is disjoint with $c_W$), then 
$M_3=\partial W\setminus l(t_-, t_+)$. Here $t_{\pm}\in (0,1)$
are the closest points to $t_0$ of the form $p'/q'\not=0$ with $q'\le q+1$
from the left and from the right of $t_0$
(note that $t_{\pm}$ exist since $1/(q+1)<t_0<q/(q+1)$).

The set $R_0\cup M_1\cup M_2\cup M_3$ is connected, closed, unbounded,
and does not separate the plane. So, $D(W, p/q)$ is a
simply-connected domain. Let us show that the domain 
$V^{p/q}=W\cup W(p/q)^*\cup V$ is contained in $D(W, p/q)$. Indeed,
$W$ and $W(p/q)^*$ are disjoint with $R_0$ as well as with $M_i, i=1,2,3$.
Also, by the condition (a), $V$ is disjoint with
$M_2$ and, by the condition (b), $V$ is disjoint
with $M_1\cup R_0$. To show that $V$ is dosjoint with $M_3$,
it is enough to prove that, for every $r'$, $0<r'<r$, the domain 
$V(r')=\psi_W^{\log}(B(2\pi i p/q, r'))$ is disjoint with $M_3$.
Fix such $r'$. Let us use the condition that
$\psi$ has a 
univalent extension to $B(0,1)\cup U$.
It implies that the domain $W\cup V=\psi(B(0,1)\cup U)$ is simply-connected.
We have: $V(r')=\psi(U(r'))$, where $U(r')=\exp(B(2\pi i p/q, r'))$.
And since $r'<r$, the boundary of 
$W\cup V(r')=\psi(B(0,1)\cup U(r'))$ is a simple closed curve.
In particular,
$\overline{V(r')}\cap \partial W$ 
is the closure $\overline{l_0}$ of a single open arc 
$l_0=
\psi(\partial B(0,1)\cap U(r'))$
containing $c_0$ because otherwise there would be a common point in the
boundaries of $W$ and $V(r')\setminus \overline{W}$ 
outside of the closed arc $\overline l_0$, a contradiction with the fact
that $\partial(W\cup V(r'))$ is a simple curve.
By the condition (a), $l_0\subset l(t_-, t_+)$,
and, by the definition, $M_3=\partial W\setminus l(t_-, t_+)$.
Therefore, $V(r')\cap M_3=\emptyset$.
We have shown that $V^{p/q}\subset D(W, p/q)$. 

As we know, $\rho_{p/q}$ is well-defined
and holomorphic in the domain $D_0=\hat B\cup W(p/q)^*\cup W$,
where $\hat B$ is a small neighborhood of the common boundary
point $c_0$ of $W$ and $W(p/q)^*$. 
As $W(p/q)^*$, $W$ are subsets of $D(W, p/q)$ and $\hat B$ is small, one
can assume that $D_0\subset D(W, p/q)$.
Since $V^{p/q}\subset D(W, p/q)$, 
the part I follows immediately from a general
\begin{prop}\label{rhocont}
The function
$\rho_{p/q}$ extends from $D_0$ to a holomorphic function defined in 
the domain $D(W, p/q)$, 
and, for each $\lambda\in \overline{B(0,1)}$ there is a unique
$c_\lambda\in D(W, p/q)$, such that $\rho_{p/q}(c_\lambda)=\lambda$.
Besides, $c_\lambda\in \overline{W(p/q)}$.
\end{prop}
\begin{proof}
The function $\rho_{p/q}$ has an analytic continuation along every
curve starting at $c_0$, which does not contain a parameter $c$,
such that $f_c$ has a periodic orbit of period $nq$ with multiplier $1$.
We will call such parameters $c$ {\it suspicious}. 
Every suspicious point lies in the boundary of $M$.
Denote by $I$ the set of those suspicious points, which are outside
of the wake $W(p/q)^*$. 
Let us prove that 
\begin{equation}\label{I}
I\cap D(W, p/q)=\emptyset.
\end{equation} 
Assume the contrary, i.e., there is $c\in D(W, p/q)\setminus W(p/q)^*$, 
such that $f_c$ has a periodic orbit of period $nq$ with multiplier $1$.
Since $c\in \partial M\cap D(W, p/q)\setminus W(p/q)^*$, 
then either (1) $c\in l(t_-, t_+)$, or
(2) $c\in L(W, p'/q')$, where $p'/q'\not=p/q$ and $q'>q+1$. 
As $f_c$, for $c\in \partial W$, has a neutral periodic
orbit of period $n$, the case (1) is excluded. The case (2)
is excluded by the above Proposition~\ref{combin}
(with $Q=q$). Thus~(\ref{I}) holds.

Denote the closure of $W(p/q)^*$ by $K$.
The function $\rho_{p/q}$ has a holomorphic extension from $W(p/q)^*$
to a small neighborhood $S$ of $K$,
because $\partial K\cap M=\{c_0\}$ and $\rho_{p/q}$ is holomorphic
in the neighborhood $\hat B$ of $c_0$. One can assume that 
$S\subset D(W, p/q)$.
Since $K$ is an unbounded continuum not separating the plane, 
$D'=D(W, p/q)\setminus K$ is a simply-connected 
domain, too. By~(\ref{I}), 
$D'$ does not contain any suspicious point.
Hence, the function $\rho_{p/q}$, which is holomorphic
in a subdomain $S\setminus K$ of $D'$, has an analytic
continuation along every curve in $D'$.
By the Monodromy Theorem,
$\rho_{p/q}$ has a well-defined analytic continuation $\tilde \rho_{p/q}$
from $S\setminus K$ to 
$D'$, and, by the Uniqueness Theorem, $\tilde \rho_{p/q}$ on $D'$ and 
$\rho_{p/q}$ on $K$ define
an analytic continuation of $\rho_{p/q}$ to $D(W, p/q)=D'\cup K$.

Thus we have shown that the function $\rho_{p/q}$
extends from $D_0$ to a holomorphic function defined 
in the domain $D(W, p/q)$.
As $\rho_{p/q}: \overline{B(0,1)}\to \overline{W(p/q)}$ is
a homeomorphism and
$|\rho_{p/q}|>1$ in $(\overline{W(p/q)^*}\setminus \overline{W(p/q)})\cup W$, 
then, {\it for every $|\lambda|\le 1$ there is one and only one
$c_{\lambda}$ in $\overline{W(p/q)^*}\cup W$, such that 
$\rho_{p/q}(c_\lambda)=\lambda$. Moreover,} 
$c_\lambda\in \overline{W(p/q)}$.
Denote $\tilde D=D(W, p/q)\setminus (\overline{W(p/q)^*}\cup W)$.
It remains to show that $|\rho_{p/q}|>1$ in $\tilde D$.
Observe that, for every $c\in D(W, p/q)$, except for perhaps
finitely many $c$ (for which $c$ is the root of a non-primitive
$nq$-hyperbolic components), $\rho_{p/q}$ is the multiplier
of some periodic orbit of $f_c$ of period $nq$. 
Assume now, by a contradiction, that, for some $|\lambda|\le 1$ and
$c_1\in \tilde D$, $\rho_{p/q}(c_1)=\lambda$.
Then $c_1$ lies in the closure of some $nq$-hyperbolic component $W_1$. 
Since $c_1\notin M_1\cup L(W, p/q)$, then $W_1$ lies in a limb
of $W$ other than its $p/q$-limb. Hence, for some $p'/q'\not=p/q$, 
$c_1\in \overline{W_1}\subset L(W, p'/q')\cup \{c(W, p'/q')\}$.
We have: $q'>q+1$, because otherwise $c_1\in M_2$. 
On the other hand, consider the root $\tilde c$ of $W_1$.
Then the map $f_{\tilde c}^{nq}$ has a fixed point with multiplier $1$.
By Proposition~\ref{combin}, $q'\le q+1$, 
which is a contradiction, because $q'>q+1$.
\end{proof}

Let us pass to the proof of Part II.
It has three main ingredients.
The first one is the inequality~(\ref{inside}).
The second one is Proposition~\ref{mon} below, which follows, for example, 
from the theory of quasinormal families due to 
P. Montel~\cite{Mon}. A family of holomorphic functions
in a domain $D$ is called quasinormal in $D$, if every
sequence of maps of the family contains a subsequence, which
converges locally uniformly in $D$ except for, possibly, 
a finite number of points.
The points where the convergence is not locally 
uniform are called irregular points.
If the number of irregular points is always at most $N$, the family is called
quasinormal of order at most $N$ (if $N=0$, the family is normal).
It is clear from the Maximum Principle, that outside of the
irregular points the sequence converges locally uniformly to infinity. 
Montel~\cite{Mon} proves the following main criterion of quasinormality:
{\it A family of functions which are holomorphic in a domain,
where it takes at most $N_0$ times the value $0$ and at most
$N_1$ times the value $1$ is quasinormal in the domain
of order at most the minimum of $N_0, N_1$.}
As an immediate corollary, we have:
\begin{prop}\label{mon}
Assume that $\{g\}$ is a family of functions which are holomorphic
in the unit disk $B(0,1)$, and such that:
(1) $g(z)=0$ if and only if $z=0$, (2) there exists $Z\not=0$,
so that each $g$ of the family takes the value $Z$ in at most
one point, and (3) there exists $z_0\not=0$, so that the set 
$\{g(z_0)\}$ is bounded. Then the family $\{g\}$ is uniformly bounded
on any compact in the unit disk.
\end{prop}
The last ingredient is the following statement, 
which is Lemma I from~\cite{Ne}:
\begin{prop}\label{unival}
Let $\omega(z)=\alpha z+c_2 z^2+...$ be regular and $|\omega(z)|\le |z|$
for $|z|<1$, and let $\omega(z)$ further satisfy $\omega(z)\not=0$ for
$0<|z|<1$; then $\omega(z)$ is univalent inside the circle $|z|=R(\alpha)$
with $R(\alpha)=1+\log(1/|\alpha|)-[(1+\log(1/|\alpha|))^2-1]^{1/2}$.
\end{prop}
Let us turn to the proof of II.
Consider $\Psi=\rho_{p/q}\circ \psi\circ \exp$.
It is holomorphic in $B$ and such that 
$\Psi(2\pi i p/q)=1$ and $\Psi'(2\pi i p/q)=-q^2$.
Define
\begin{equation}\label{normal}
g(w)=\frac{\Psi(rw+2\pi i p/q)-1}{r q^2}.
\end{equation}
Then $g$ is holomorphic in the unit disk, moreover, $g(0)=0$
and $g'(0)=-1$. 
We are going to show that $g$ satisfies the conditions (1)-(3) of
Proposition~\ref{mon} with $Z=-1/(2X)$ and with $z_0=-1/2$.

(1) Assume $g(w)=0$. It means that $\rho_{p/q}(c)=1$, for some $c\in V$.
By the part I, then $c=c_0$, i.e. $w=0$.

(2) Assume $g(w)=-1/(2X)$. 
Then, for some $c_1\in V$, $\rho_{p/q}(c_1)=1-r q^2/(2X)$.
Since $r q^2<X$, the point $\lambda=1-r q^2/(2X)\in (1/2, 1)$,
i.e., it lies in the unit disk. 
Hence, by Part I, such $c_1$ is unique, that is, 
$w$ is the only solution of the equation $g(w)=-1/(2X)$.

(3) Let us fix $w=-1/2$. Now,
apply ~(\ref{inside}) with 
$c=\psi\circ\exp(-r/2+2\pi i p/q)$ and
$\log\rho_W(c)=-r/2+2\pi i p/q$:
$$\frac{|\log(1+r q^2 g(-1/2))|^2}{\log|1+r q^2 g(-1/2))|}
<q^2\frac{|-r/2|^2}{-(-r/2)}=\frac{1}{2} r q^2.$$
Geometrically, it means that the point $1+\delta g(-1/2)$ belongs
to the set $E=\{\exp(z): |z-\delta/4|<\delta/4\}$, where
$\delta=r q^2$. 
If $\delta<\delta_0$,
where $\delta_0$ is a small fixed number, then
$E$ is contained in a small disk around $1$ of radius at most $\delta$.
Hence, $|g(-1/2)|<1$. On the other hand, if $\delta\ge \delta_0$ 
and $\delta<X$, then 
$|g(-1/2)|<(1+\exp(\delta))/\delta\le (1+\exp(X))/\delta_0$.

We have checked that the conditions (1)-(3) hold for
every $g$ as above. Therefore, by Proposition~\ref{mon},
for every $X>0$ there is $C$, such that
$|g(w)|<C$, for $|w|<9/10$. Now we can apply
Proposition~\ref{unival} to the function
$\omega(z)=g(9 z/10)/C$, $|z|<1$. We get that $g$ is univalent
in the disk $|w|<\Lambda:=(9/10)R(9/(10 C))$. It means
that $\Psi$ is univalent in the disk $B(2\pi i p/q, \Lambda r)$.
By the classical distortion bounds for univalent maps,
there exists $\Lambda_0$, which depends on $\Lambda$ only, such that
$2/3<|\Psi'(L)/\Psi'(2\pi i p/q)|<3/2$,
for $L\in B(2\pi i p/q, \Lambda_0 r)$. Since 
$|\Psi'(2\pi i p/q)|=q^2$, this proves the covering property.

To complete the prove of Part II, let us show that
the function $\psi_{W(p/q)}$ has a 
univalent extension to $B(0,1)\cup B(1,2 q^2 \Lambda_0 r/3)$.
Indeed, by what we have just proved, the function 
$\rho_{p/q}^{-1}=\psi_W^{\log}\circ \Psi^{-1}$
is well-defined and univalent in $B(1,2 q^2 \Lambda_0 r/3)$.
In other words, $\psi_{W(p/q)}=\rho_{p/q}^{-1}$ 
has an analytic continuation
from $B(0,1)$ to the domain $B':=B(0,1)\cup B(1,2 q^2 \Lambda_0 r/3)$
(denote this continuation again by $\psi_{W(p/q)}$).
In order to show that $\psi_{W(p/q)}$ is univalent in $B'$,
observe that $\psi_{W(p/q)}(B')\subset W(p/q)\cup V\subset V^{p/q}$,
and, by Part I, $\rho_{p/q}$ is a holomorphic function in $V^{p/q}$.
Thus, $\psi_{W(p/q)}$ in $B'$ has a well-defined inverse function 
$\rho_{p/q}$. The proof of Lemma is completed. 
\begin{com}\label{lambda0}
Obviously, Lemma~\ref{q} remains valid if the constant $\Lambda_0$ 
is replaced by a smaller one.
Using the classical bounds 
$\frac{1-|z|}{(1+|z|)^3}\le |f'(z)|\le \frac{1+|z|}{(1-|z|)^3}$
for any univalent function $f(z)=z+...$ in the unit disk 
(see e.g.~\cite{Gol}), 
it is easy to check that, by the optimality of the bounds,
$\Lambda_0<\Lambda/8$, and 
one can take $\Lambda_0=\Lambda/16$. 
\end{com}

In the next two Lemmas we give some bounds on $r$ from below
in terms of $n$ and $p/q$ to satisfy
the conditions (a) and (b) of the above Lemma~\ref{q}. 
\begin{lem}\label{r}(cf.~\cite{Che})
Let $n\ge 1$ and $p/q\not=0$. 
Given an $n$-hyperbolic component $W$,
consider the corresponding function $\psi_W^{\log}$.
Assume that $\psi_W^{\log}$ extends to a function, which is defined and
univalent on the disk $B=B(2\pi i p/q, 1/(2n q^3))$. Then 
the domain $V=\psi_W^{\log}(B)$ is disjoint with 
any limb $L(W, p'/q')$ other than $L(W, p/q)$ and such that $q'\le q+1$. 
\end{lem}
\begin{proof}
The image of any limb $L(W,p'/q')$ by the function $\log\rho_W$
is contained in Yoccoz's circle
$Y_n(p'/q')=\{L: |L-(2\pi i p'/q'+\frac{n\log 2}{q'})|<\frac{n\log 2}{q'}\}$.
Let us estimate the distance $d$ between the point $2\pi i p/q$ and the circle
$Y_n(p'/q')$, where $p'/q'\not=p/q$ and $q'\le q+1$. 
Then $|p'/q'-p/q|\ge 1/(q q')$ and, hence,
$$d=[(2\pi(\frac{p'}{q'}-\frac{p}{q}))^2+(\frac{n\log 2}{q'})^2]^{1/2}-
\frac{n\log 2}{q'}\ge \frac{1}{q'}\{[\frac{4\pi^2}{q^2}+(n\log 2)^2]^{1/2}
-n\log 2\}\ge$$
$$\frac{4\pi^2}{q^2(q+1)}\frac{1}{[\frac{4\pi^2}{q^2}+(n\log 2)^2]^{1/2}
+n\log 2}\ge \frac{1}{2n q^3},$$
for all $n\ge 1$ and $q\ge 2$.
\end{proof}
\begin{lem}\label{rb}
Let $n\ge 1$ and $p/q\not=0$. 
Set 
$$\hat r=\min\{\frac{1}{n}(\frac{p}{q})^2, \ 
\frac{1}{2^{n/2}}|\frac{p}{q}|\}.$$
Given an $n$-hyperbolic component $W$,
consider the corresponding function $\psi_W^{\log}$.
Assume that $\psi_W^{\log}$ extends to a function, which is defined
and univalent on the disk $B=B(2\pi i p/q, \hat r)$. Then 
the domain $V=\psi_W^{\log}(B)$ is contained in the wake $W^*$.
In particular, the condition (b) of Lemma~\ref{q} holds.
\end{lem}
\begin{proof}
Each point of the periodic orbit $O_W(c)$ of period $n$
has an analytic continuation 
from $W$ to the wake $W^*$ and, by continuity, to 
$\overline{W^*}\setminus \{c_W\}$.
(In fact, one can extend it analytically, for example, to
${\bf C}\setminus (\{c>1/4\}\cup (M\setminus W^*))$.)
Assume that $V$ is not contained in $W^*$. Then there is 
$c_1\in V\setminus M$, such that
$c_1$ lies on a parameter ray of argument $t_{c_1}\in \{t_{\pm}(W)\}$.
Consider now the dynamical plane of $f_{c_1}$ and its dynamical rays
(for the definiton of dynamical rays to the disconnected Julia set,
see e.g.~\cite{levprz} and Subsection~\ref{apdis} of Appendix).
Then a point of the periodic orbit $O_W(c_1)$ of $f_{c_1}$ of period $n$
must be the landing point
of a (non-smooth) dynamical ray of $f_{c_1}$ of argument $t_{c_1}$
(see e.g.~\cite{Leyo}).
As $t_{c_1}$ is a periodic point of the doubling
map $\sigma$ of period $n$,
we can employ Corollary~\ref{zero} of Appendix. 
Denote by $\rho_1$ the multiplier of the periodic orbit $O_W(c_1)$.
By the condition of the present Lemma, 
$\log\rho_1$ lies in the disk $B$. In particular,
$\log|\rho_1|<\hat r$.
On the other hand, 
by Corollary~\ref{zero}, we obtain that a branch of $\log\rho_1$ lies
in the circle 
$Y=Y_n(0, \delta(n, (1/n)\log|\rho_1|))$.
Therefore,
the disks $B$ and $Y$ must intersect. Let us show that this is impossible.
Indeed, by~(\ref{diszero})-(\ref{rnrho}) of Corollary~\ref{zero},
and since $\log|\rho_1|<\hat r$, 
then $Y\subset B(R,R)$, where
$$R=\frac{n\pi\log 2}{\arctan\frac{n\pi}{(2^n-1)\hat r}}.$$
We need to check that
$(2\pi p/q)^2+R^2>(\hat r + R)^2$, or $(2\pi p/q)^2>\hat r^2 + 2R\hat r$.
Denote $\beta=n\pi/((2^n-1)\hat r)$.
Consider two cases. If $\beta\ge 1$, then 
$R<4n\log 2$, and 
$$\hat r^2 + 2R\hat r\le
(\frac{1}{n})^2(\frac{p}{q})^4+\frac{1}{n}(\frac{p}{q})^2 
8n\log 2\le (\frac{p}{q})^2(1+8\log 2)<(2\pi \frac{p}{q})^2.$$
If $\beta<1$, then we use that $\arctan \beta>\beta/2$ and, hence,
$$R<\frac{n\pi(\log 2) 2(2^n-1)\hat r}{n\pi}=2\hat r (2^n-1)\log 2.$$
Thus 
$$\hat r^2 + 2R\hat r<\hat r^2 (1+4(2^n-1)\log 2)
\le (\frac{p}{q})^2 \frac{1}{2^n}(1+4(2^n-1)\log 2)<(2\pi \frac{p}{q})^2.$$
\end{proof}

We can now combine Lemmas~\ref{r}-\ref{rb}
with Lemma~\ref{q} 
and get an effective version of Lemma~\ref{q}.
As $1/(2n q^3)\le (1/n)(p/q)^2$, the minimum of radii
introduced in Lemmas~\ref{r}-\ref{rb} is the number $r(n, p/q)$
defined by
\begin{equation}
r(n, p/q)=\min\{\frac{1}{2n q^3}, \ \frac{1}{2^{n/2}}|\frac{p}{q}|\}.
\end{equation}
Put $r=r(n, p/q)$ in Lemma~\ref{q}. 
Since $r q^2\le 1/(2nq)<1$, we can apply Lemma~\ref{q} with $X=1$
and find corresponding $0<\Lambda_0<\Lambda<1$.
Then, by Lemma~\ref{q} (I)-(II), the 
function $\Psi_{W,p/q}$ is well-defined in 
$B(2\pi i p/q, r(n, p/q))$ and univalent in
$B(2\pi i p/q, \Lambda r(n, p/q))$.
Moreover, 
\begin{equation}\label{kuku}
B(1, \frac{2}{3} q^2\Lambda_0 r(n, \frac{p}{q}))\subset 
\Psi_{W,p/q}(B(2\pi i \frac{p}{q}, \Lambda_0 r(n, \frac{p}{q}))\subset
B(1, \frac{3}{2} q^2\Lambda_0 r(n, \frac{p}{q})).
\end{equation}
As $\frac{3q^2\Lambda_0 r(n, p/q)}{2}<\frac{3}{4nq}<1$, a branch of $\log$
is well-defined in $B(1, 3q^2\Lambda_0 r(n, p/q)/2)$, such that it
vanishes at the point $1$. 
Therefore, using this branch of $\log$, the function
\begin{equation}\label{lambda}
\lambda_{W,p/q}:=\log\Psi_{W(p/q)}=
\log\circ \rho_{W(p/q)}\circ \psi_W\circ \exp
\end{equation}
is well-defined and univalent in 
$B(2\pi i p/q, \Lambda_0 r(n, p/q))$
and vanishes only at the point $2\pi i p/q$.
Next, we apply the derivative estimate for $\Psi_{W(p/q)}$, see 
Lemma~\ref{q} (II).
Using~(\ref{kuku}) and that $\Lambda_0<1/8$ (see Comment~\ref{lambda0})
we easily conclude from the definition of $\lambda_{W,p/q}$ 
that $q^2/2<|\lambda_{W, p/q}'|<2 q^2$ in 
$B(2\pi i p/q, \Lambda_0 r(n, p/q))$.
Finally, by the last conclusion of Lemma~\ref{q}(II) the function
$\psi_{W(p/q)}$ has a univalent extension to 
$B(0,1)\cup B(1, 2 q^2 \Lambda_0 r(n, p/q)/3)$. 

We define a constant $Q_0$ by the condition:
for $q>Q_0$ and $0<\epsilon\le 1/(4q)$,
\begin{equation}\label{Q0}
\exp(B(0, \epsilon))\subset B(1, \frac{4}{3}\epsilon).
\end{equation}
We use this with $\epsilon=q^2 \Lambda_0 r(n, p/q)/2<1/(4q)$,
where $q>Q_0$.

We get the part A of the following statement, which will serve as 
an indunction argument in the proof of the Main Lemma~\ref{notlc}.

Recall that $\rho_0=\exp(2\pi i p/q)$ and $c_0=\psi_W(\rho_0)$.
\begin{lem}\label{q^{-3}}
There exists $0<\Lambda_0<1$ as follows.
Assume that
$\psi_W$ has a univalent extension to
$B(0,1)\cup U$, where $U=\exp(B(2\pi i p/q, r(n, p/q)))$.
Then A-B hold.

A. The function $\Psi_{W,p/q}=\rho_{W(p/q)}\circ \psi_W\circ \exp$
is holomorphic in $B(2\pi i p/q, r(n, p/q))$, and 
it is equal to $1$ only at the center $2\pi i p/q$.
Furthermore, the function $\lambda_{W,p/q}$ introduced above:
$$\lambda_{W,p/q}=\log\circ \rho_{W(p/q)}\circ \psi_W\circ \exp$$
is well-defined and univalent in $B(2\pi i p/q, \Lambda_0 r(n, p/q))$, and 
it is equal to zero only at the center $2\pi i p/q$. 
For every $L\in B(2\pi i p/q, \Lambda_0 r(n, p/q))$, we have:
$\frac{1}{2}q^2<|\lambda_{W, p/q}'(L)|<2 q^2$.
In particular, 
the inverse function $\lambda_{W, p/q}^{-1}$
is defined and univalent in $B(0,\Lambda_0 q^2 r(n, p/q)/2)$, and
$$\frac{1}{2q^2}<|(\lambda_{W,p/q}^{-1})'(L)|<\frac{2}{q^2}$$
there.
Finally, if $q>Q_0$, the function $\psi_{W(p/q)}$ has a univalent extension
to $B(0,1)\cup \hat U$, where 
$\hat U=\exp(B(0, \Lambda_0 q^2 r(n, p/q)/2))$.

B.  
For each $k=1,...n$, there exists a function $F_k$, which
is defined and holomorphic in the disk 
$$S=\{|s|<(\frac{1}{2} r(n, p/q))^{1/q}\},$$ 
such that $F_k(0)=0$, $F_k'(0)\not=0$, and the following holds.
For every $s\in S$, $s\not=0$, 
and corresponding $\rho=\rho_0+s^q$, the points 
$b_k(c_0)+F_k(s\exp(2\pi i \frac{j}{q}))$, $k=1,...,n, j=0,...,q-1$,
form the periodic orbit $O^{p/q}(c)$ of $f_c$ of period $nq$, 
where $c=\psi_W(\rho)$. The multiplier of $O^{p/q}(c)$ is equal to
$\rho_{W(p/q)}(c)$.
\end{lem}

\

It remains to prove the part B.
It holds locally, in a neighborhood of the point $s=0$
(see the discussion in the beginning of Section~\ref{bif},
where the functions $F_k$ were introduced).
Let us show that the functions $F_k$ don't have singularities
in the disk $S$. 
Indeed, otherwise being continued along a curve in $S$, which starts
at $s=0$ and ends at some $s_1\not=0$,
the multiplier $\rho_{W(p/q)}$
of $O^{p/q}$ becomes equal to $1$, for some $c_1\not=c_0$.
Here $c_1=\psi_W(\rho_1)$, where $|\rho_1-\rho_0|=|s_1|^q<r(n, p/q)/2$. 
Besides, $\rho_1\not=\rho_0$.
It is easy to check that then a branch of $\log\rho_1$
is contained in $B(2\pi i p/q, r(n, p/q))$, 
and $\log\rho_1\not=\log\rho_0=2\pi i p/q$.
On the other hand, by the part A, 
$\rho_{W(p/q)}\circ \psi_W\circ \exp$ 
takes the value $1$ in $B(2\pi i p/q, r(n, p/q))$
only at the point $2\pi i p/q$,
a contradiction.

\subsection{The function $H$}\label{H}
\begin{defi}\label{HH} Define
a real strictly increasing
smooth function 
$$H: [0,1)\to [0, \infty), \ \ H(0)=0$$ 
as follows. Let $G$ be the set of all
holomorphic functions $g: B(0,1)\to {\bf C}\setminus \{1\}$, 
such that $g(w)=0$ if and only if $w=0$.
Then
$$H(u)=\sup\{ |g(w)|: |w|\le u, g\in G\}.$$ 
\end{defi}

There is an explicit expression for $H$.
It is obtained as follows.
Let $J(w)$ be a holomorphic function in $B(0,1)$, such that
$J(0)=0$, $J'(0)>0$, and 
$$J: B(0,1)\setminus \{0\}\to {\bf C}\setminus \{0,1\}$$
is an infinite unbranched cover. Such function is investigated in
~\cite{Ne}, see also ~\cite{Er},~\cite{Hur},~\cite{Nebook}.
By the Schwarz lemma, 
$H(u)=\max_{|w|=u} |J(w)|$.
On the other hand, by~\cite{Ne},
$J(w)=16w \Pi_{k=1}^\infty \frac{(1+w^{2k})^8}{(1+w^{2k-1})^8}$.
(Apparently, $J$ is equal to the square of the so-called elliptic modulus,
see e.g.\cite{Nebook}.)
Thus,
$H(u)=-J(-u)=16u \Pi_{k=1}^\infty \frac{(1+u^{2k})^8}{(1-u^{2k-1})^8}$.
Using a bound for $|J(w)|$ proved in~\cite{Ne}, we get:
\begin{equation}\label{bh}
H(u)\le \frac{1}{16}\exp{(-\pi^2/\log u)}.
\end{equation}
\subsection{Main Lemma}
Theorem~\ref{intronlc} will be a consequence of
Lemma~\ref{notlc} below and a renormalization argument.

It is easy to check the existence of 
$Q_1$, such that, for $q>Q_1$ and $|p/q|\le 1/2$,
$$\exp(B(2\pi i \frac{p}{q}, \frac{1}{2q^3}))\subset \{\rho: Re(\rho)<1\}.$$
Recall that $r(n, p/q)=\min\{1/(2n q^3), |p/q| 2^{-n/2}\}$,
and the constant $Q_0$ is defined by
the condition~(\ref{Q0}) before Lemma~\ref{q^{-3}} .
\begin{lem}\label{notlc}
Let $\Lambda_0$ be the constant 
from Lemma~\ref{q^{-3}}.
Set 
$$\alpha=\frac{\Lambda_0}{4}, \ \ \ 
Q=\max \{Q_0, Q_1\}.$$
Let $t_0, t_1,...,t_m$,... be a sequence
of rational numbers $t_m=p_{m}/q_{m}\in (-1/2, 1/2]\setminus \{0\}$.
Let $W^0$ be the main cardioid.
Denote $W^{m}=W^{m-1}(t_{m-1})$, $m=1,2,...$,
in other words, the closure of the
hyperbolic component $W^m$ touches the closure of the hyperbolic
component $W^{m-1}$ at the point 
$c_{m-1}:=c(W^{m-1}, t_{m-1})$
with the internal argument
$t_{m-1}$. 
Denote 
$$n_0=1, \ \ \ n_m=q_0q_1...q_{m-1}, \ m>0,$$ 
i.e.,
$n_m$ is the period of the attracting periodic orbit
of $f_c$, for $c\in W^m$.

(C1) Assume that, for $m\ge 0$, we have: $q_m>Q$ and
\begin{equation}\label{Yrel1}
|t_0|<\frac{\alpha}{8\pi}, \ \ 
|t_m|<\frac{\alpha}{4\pi}q_{m-1}^2 r(n_{m-1}, \frac{p_{m-1}}{q_{m-1}})=
\frac{\alpha}{4\pi}\min\{\frac{1}{2 n_m}, \frac{|p_{m-1}q_{m-1}|}
{2^{n_{m-1}/2}}\}, m\ge 1.
\end{equation}
Then the sequence $c_m$ converges to some $c_\ast\in \partial M$.  

(C1') If, additionally, 
\begin{equation}\label{Yrel1'}
\limsup_{m\to \infty} \frac{q_m}{\max\{n_{m}^2, n_{m-1}2^{n_{m-1}/2}\}}>
\frac{8}{\alpha},
\end{equation}
then the Mandelbrot set is locally connected at $c_\ast$.

(C2) Assume that the conditions of (C1) hold, and assume also that
\begin{equation}\label{exp1}
|t_0| H(u_0)+\sum_{m=0}^\infty  \frac{u_{m}}{q_m(1-u_{m})}H(u_{m+1})
<\frac{1}{100}, 
\end{equation}
where 
\begin{equation}\label{Rintro}
u_{m}=|32 t_{m+1}\max\{2n_{m+1}, 
\frac{2^{n_m/2}}{|p_m q_m|}\}|^{1/q_m}, m\ge 0.
\end{equation}
Then the map $f_{c_\ast}$ is infinitely renormalizable
with non locally connected Julia set.
\end{lem}
\begin{proof}
Introduce some notations. Let
$$r_m=r(n_m, \frac{p_m}{q_m})=
\min\{\frac{1}{2n_m q_m^3}, \ \frac{1}{2^{n_m/2}}|\frac{p_m}{q_m}|\}, 
\ m=0,1,...$$
$$B_0:=B(0, \alpha), \ \ 
B_m:=B(0, \alpha q_{m-1}^2 r_{m-1}),\  m=1, 2,...,$$ 
$$\tilde B_m:=B(2\pi i p_m/q_m, r_m), \ \ \ U_m=\exp(\tilde B_m), 
\ m=0,1,....$$ 
As usual, $\psi_{W^m}$ is the function,
which is inverse to the multiplier function $\rho_{W^m}$. 
Using notations of Lemma~\ref{q^{-3}}, denote
$$\psi_m=\psi_{W^m}^{\log}=\psi_{W^{m}}\circ \exp, \  \ \
\lambda_m=\lambda_{W^m, t_m}=
\log\circ \rho_{W^{m+1}}\circ \psi_{W^m}\circ \exp.$$
For $m\ge 0$, 
$\lambda_m$ is defined and holomorphic
in a neighborhood of the point $2\pi i t_m$, and a branch of the
log is chosen so that $\lambda_m(2\pi i t_m)=0$.
Note that $c_m=\psi_m(2\pi i t_m)$.
We have explicitely $\psi_{W^0}(\rho)=\rho/2-(\rho/2)^2$, so that
$\psi_{W^0}$ is holomorphic in the plane and is univalent in 
any domain, which does not contain two points $\rho_1\not=\rho_2$ with
$\rho_1+\rho_2=2$, particularly, it is univalent in the half-plane
$\{Re(\rho)<1\}$. By the choice of $Q_1$ and $Q$, $\psi_{W^0}$
is univalent in $B(0,1)\cup U_0$. This will allow us to start
applying Lemma~\ref{q^{-3}}.
Now, let us verify that, for $m\ge 0$,
\begin{equation}\label{disks}
\tilde B_m\subset B_m\setminus{0}.
\end{equation}
Indeed, 
$$\frac{2\pi |p_m/q_m|}{r_m}\ge\frac{2\pi |p_m/q_m|}{1/(2n_m q_m^3)}\ge 
16\pi>1.$$
Hence, if $\hat r_m$ denotes the radius of $B_m$ and using
the condition (\ref{Yrel1}), 
$$2\pi\frac{|p_m|}{q_m}+r_m<4\pi\frac{|p_m|}{q_m}<\hat r_m.$$
This proves~(\ref{disks}).

Note that $\psi_0=\psi_{W^0}\circ \exp$
is holomorphic in ${\bf C}$. Set $C=\max\{|(\psi_0)'(L)|: L\in B_0\}$.

{\bf Claim 1.}
{\it For $m=1,2,3,...$, the following holds.

(i) The function $\psi_m$ extends to a
univalent map defined in $B_m$. Moreover,
the univalent function $\psi_{W^m}: B(0,1)\to W^m$ has a univalent extension 
to $B(0,1)\cup \exp(B_m)$.

(ii)  For $L\in B_m$, 
$$|\psi_m'(L)|\le C\frac{2^{m}}{n_m^2}.$$

(iii) For any $k\ge 0$, denote 
$$R_k=\psi_k(B_k),$$
(so that $c_k\in R_k$). Then
$R_{m}\subset R_{m-1}$, and the diameter of the set
$R_m$ is less than $2C\alpha 2^{m}/n_m^3$. In particular,
$\{R_m\}$ shrink to a point $c_\ast$, which is the limit
of the sequence $c_m$.

(iv) If, for some $m>0$, 
\begin{equation}\label{iv}
\frac{q_m}{\max\{n_{m}^2, n_{m-1}2^{n_{m-1}/2}\}}>
\frac{8}{\alpha},
\end{equation}
then the limb $L(W^m, t_m)$ is contained in $R_m$.}

We prove (i)-(iii) of the Claim 1 by induction in $m=1,2,3,...$.
The proof is an almost straighforward application of Lemma~\ref{q^{-3}}.
We have:
\begin{equation}\label{svyaz}
\psi_{m+1}=\psi_{m}\circ \lambda_{m}^{-1}, \ m\ge 0,
\end{equation}
whenever the right hand-side is defined.

Let us prove (i)-(iii) for $m=1$.
As we checked before the Claim 1, $\psi_{W^0}$ has a univalent extension
to $B(0,1)\cup U_0$, that is, the conditions of Lemma~\ref{q^{-3}} hold,
for $W=W^0$, and $p_0, q_0$ instead of $p, q$.
By Lemma~\ref{q^{-3}} (A), the inverse $\lambda_0^{-1}$
is defined and univalent in $B_1=B(0,(\Lambda_0/4)q_0^2 r_0)$, 
and $\lambda_0^{-1}(B_1)$ is contained
in $\tilde B_0$. 
Also, for
$L\in B_{1}$,
$\frac{1}{2q_0^2}<|(\lambda_0^{-1})'(L)|<\frac{2}{q_0^2}$.
By~(\ref{svyaz}), $\psi_{1}$ extends to a univalent map defined in
$B_{1}$, and, by the Chain Rule,
for $L\in B_{1}$,  
$$|\psi_{1}'(L)|\le C\frac{2}{q_0^2}=
C\frac{2}{n_{1}^2}.$$
The property (iii) follows from (i)-(ii):
$R_{1}=\psi_1(B_{1})=\psi_{0}\circ \lambda_{0}^{-1}(B_{1})
\subset \psi_0(B_0)=R_0$. Finally, by the last conclusion
of Lemma~\ref{q^{-3}} (A), $\psi_{W^1}$ has a univalent extension to 
$B(0,1)\cup \exp(B_1)$. 

Step of induction. This is an obvious modification
of the argument for $m=1$. So, assume (i)-(ii)
hold for $m\ge 1$. By (\ref{disks}) and (i), 
the conditions of Lemma~\ref{q^{-3}} hold,
for $W=W^m$ and $p_m, q_m$ instead of $p, q$.
Therefore, the inverse $\lambda_m^{-1}$
is defined and univalent in $B_{m+1}=B(0,\alpha q_{m}^2 r_{m})$, 
also 
\begin{equation}\label{recent1}
\lambda_m^{-1}(B_{m+1})\subset \tilde B_m\subset B_m,
\end{equation} 
and, for
$L\in B_{m+1}$,
\begin{equation}\label{der}
\frac{1}{2q_m^2}<|(\lambda_m^{-1})'(L)|<\frac{2}{q_m^2}.
\end{equation}
By~(\ref{svyaz}), $\psi_{m+1}$ extends to a univalent map defined in
$B_{m+1}$, and, by the Chain Rule and the induction assumption,
for $L\in B_{m+1}$,  
$$|\psi_{m+1}'(L)|\le C\frac{2^{m}}{n_m^2}\frac{2}{q_m^2}=
C\frac{2^{m+1}}{n_{m+1}^2}.$$
In turn,
$R_{m+1}=\psi_{m+1}(B_{m+1})=\psi_{m}\circ \lambda_{m}^{-1}(B_{m+1})
\subset \psi_m(B_m)=R_m$, i.e.
(iii) holds as well, for $m+1$.
Finally, by the last conclusion 
of Lemma~\ref{q^{-3}} (A), $\psi_{W^{m+1}}$ has a univalent extension to 
$B(0,1)\cup \exp(B_{m+1})$. 

The induction is completed.
Let us prove (iv). Assuming the condition of (iv) holds,
let us check that 
Yoccoz's circle $Y_{n_m}(t_m)$ is contained in $B_m$.
Indeed, $Y_{n_m}(t_m)$ is contained in the disk centered at $2\pi i p_m/q_m$
of radius $2n_m\log 2/q_m$, while, by~(\ref{iv}),
$$\frac{2n_m\log 2}{q_m}<
\frac{2\alpha n_m\log 2}{8\max\{n_{m}^2, n_{m-1}2^{n_{m-1}/2}\}}
\le 
\frac{\alpha\log 2}{2}\min\{\frac{1}{2n_m}, \frac{q_{m-1}}{2^{n_{m-1}/2}}\}\le
\frac{\alpha\log 2}{2}q_{m-1}^2 r_{m-1}.$$
Then, using~(\ref{Yrel1}),
$$2\pi\frac{|p_m|}{q_m}+\frac{2n_m\log 2}{q_m}<[\frac{\alpha}{2}+
\frac{\alpha\log 2}{2}]
q_{m-1}^2 r_{m-1}<
\alpha q_{m-1}^2 r_{m-1},$$
which is the radius of $B_m$. Thus  $Y_{n_m}(t_m)\subset B_m$,
hence, $L(W^m, t_m)\subset \psi_m(Y_{n_m}(t_m))\subset 
\psi_m(B_m)=R_m$. 
This finishes the proof of the Claim 1.
It yields immediately the statements (C1)-(C1') of the Lemma.

Now we pass to the proof of the statement (C2).
Let us denote by $O_k(c)$ the $n_k$-periodic orbit
of $f_c$, which is attracting if $c\in W^k$, $k=0,1,2,...$.
As we know, $O_k(c)$ extends holomorphically for $c$ in the wake
$(W^k)^*$ of $W^k$.  
Given $m\ge 1$, consider any point $b(c)$ of the periodic orbit
$O_m(c)$, for $c\in (W^{m})^*$.
By~(\ref{disks}) and by Claim 1(i), $\psi_{W^{m-1}}$ has a univalent 
extension to $B(0,1)\cup \exp(\tilde B_{m-1})$
(for $m=1$, this was checked before the Claim 1). 
Thus
the conditions of  
Lemma~\ref{q^{-3}} hold, for $W=W^{m-1}$ and $p/q=t_{m-1}$. 
Then the conclusion (B) of that Lemma tells us that the function
$b$ extends to
a neighborhood of $c_{m-1}$ in the following sense.
There is a holomorphic function $Z$ of a local parameter $s$ in the
disk 
$$S_{m-1}:=\{|s|<v_{m-1}\}, \ \ \ 
v_{m-1}=(\frac{r_{m-1}}{2})^{1/q_{m-1}},$$ 
such that it matches $b(c)$, 
i.e., $b(c)=Z(s)$ for
$$c=c_{m-1}(s):=\psi_{W^{m-1}}(\exp(2\pi i t_{m-1})+s^{q_{m-1}}).$$
Moreover, the points $Z(s e^j_{q_{m-1}})$, where $j=1,...,n_{m-1}-1$
and $e_q=\exp(2\pi i /q)$, also belong to $O_m$.
We denote by $b^+$ the point $Z(s e_{q_{m-1}})$ of $O_m$,
which is uniquely defined for $s\in S_{m-1}$.
Let us estimate the distance between $b=Z(s)$ and $b^+=Z(s e_{q_{m-1}})$.

{\bf Claim 2.} 
{\it $|Z(s)|<3$ for $s\in S_{m-1}$, $m=1,2,...$ .}

{\it Proof of Claim 2.} For $|c|<5$, every point z, such that $|z|\ge 3$,
escapes under the dynamics of $f_c$. As $Z(s)\in J_{c_{m-1}(s)}$, it is then
enough to check that $|c_{m-1}(s)|<5$. Fix $s\in S_{m-1}$ and denote
$\tilde c=c_{m-1}(s)$. If $\tilde c\in M$, then $|\tilde c|\le 2$, so we
assume $\tilde c\notin M$.
We have: $|\rho_{W^{m-1}}(\tilde c)-\exp(2\pi i t_{m-1})|=
|s|^{q_{m-1}}<r_{m-1}/2$,
where $r_{m-1}\le 1/(2 n_{m-1} q_{m-1}^3)\le 1/16$.
Hence, 
$$\log|\rho_{W^{m-1}}(\tilde c)|<\log(1+\frac{r_{m-1}}{2})<
r_{m-1}\le \frac{1}{16}.$$
On the other hand, consider the Riemann map 
$R:\{|w|>1\}\to {\bf C}\setminus M$, where $R(w)=w+\alpha_0+O(1/|w|)$ as 
$w\to \infty$.
If $R(\tilde w)=\tilde c$, then, in the notations of Subsection~\ref{apdis}
of Appendix,
$|\tilde w|=|B_{\tilde c}(\tilde c)|=\exp(2 a_{\tilde c})$.
In turn, by Theorem~\ref{disbelow} of Subsection~\ref{apdis}, 
$$a_{\tilde c}\le \frac{1}{n_{m-1}}\log|\rho_{W^{m-1}}(\tilde c)|\le 
\frac{1}{16},$$ 
that is,
$|\tilde w|\le \exp(1/8)$. By a general property of 
univalent maps (see e.g.~\cite{Gol}, Ch.II), for every $r\ge 1$, 
the complement of the image
$R(\{|w|>r\})$ belongs to the disk $|c-\alpha_0|\le 2r$.
Setting here $r=1$ and using that $-2, 1/4\in M$ and $\alpha_0\in {\bf R}$,
we get $-2<\alpha_0\le 0$ (in fact, it is well-known that $\alpha_0=-1/2$).
Then, setting $r=\exp(1/8)$, we get finally that $|\tilde c|<2+2\exp(1/8)<5$.
Claim 2 is proved.

Thus $|Z(s)|<3$ for $s\in S_{m-1}$. Therefore,
\begin{equation}\label{+}
|Z(s e_{q_{m-1}})-Z(s)|<3 \frac{2\pi}{q_{m-1}}
\frac{\frac{|s|}{v_{m-1}}}{1-\frac{|s|^2}{v_{m-1}^2}}.
\end{equation}
Let us detect $s$ for $O_m$, which corresponds to the limit
parameter $c_*$.
Using~(\ref{svyaz})-(\ref{recent1}), we have : 
$$c_*\in \psi_{m+1}(B_{m+1})=\psi_m\circ \lambda_m^{-1}(B_{m+1})\subset
\psi_m(\tilde B_m)=\psi_{m-1}\circ \lambda_{m-1}^{-1}(\tilde B_m).$$
As $\lambda_{m-1}(2\pi i t_{m-1})=0$,
then~(\ref{der}) gives us that 
\begin{equation}\label{recent2}
\lambda_{m-1}^{-1}(\tilde B_m)\subset B(2\pi i t_{m-1}, C),
\end{equation}
where
\begin{equation}\label{new}
C=\frac{2}{q_{m-1}^2}(2\pi |t_m|+r_m)\le
\frac{2}{q_{m-1}^2}(2\pi |t_m|+\frac{1}{2 n_m q_m^3})
<\frac{16 |t_m|}{q_{m-1}^2}.
\end{equation}
Note that, by~(\ref{Yrel1}), 
\begin{equation}\label{nado}
\frac{16 |t_{m}|/q_{m-1}^2}{r_{m-1}/2}<32 
\frac{\alpha}{4\pi}<\frac{2}{\pi}<1.
\end{equation}
Therefore, if 
$s_{m-1}\in S_{m-1}$ is such that $c_*=c_{m-1}(s_{m-1})$, then 
\begin{equation}\label{v}
\frac{|s_{m-1}|}{v_{m-1}}< 
(\frac{16 |t_{m}|/q_{m-1}^2}{r_{m-1}/2})^{1/q_{m-1}}=
(32 |t_m| \max\{2n_{m}, \frac{2^{n_{m-1}/2}}{|p_{m-1}q_{m-1}|}\})^{1/q_{m-1}},
\end{equation}
and is less than $1$.
Now we consider two consequtive periodic orbits
$O_{m}$ and $O_{m+1}$. 
Let $b_{m+1}(c)$ be any point of $O_{m+1}(c)$, for $c\in (W^{m+1})^*$.
By the above, there is a holomorphic function 
$Z_{m+1}$ of the local parameter $s\in S_m$, such that 
$b_{m+1}(c)=Z_{m+1}(s)$, where $c=c_m(s)$
provided $c_m(s)\in (W^{m+1})^*$.
For $s\in S_m$, we have:
$$\rho_{W^{m-1}}(c_m(s))\in \rho_{W^{m-1}}\circ \psi_{W^m}
(B(e^{2\pi i t_m}, r_m/2))\subset 
\rho_{W^{m-1}}\circ \psi_{W^m}\circ \exp (\tilde B_m).$$
Consider the set 
$\Omega:=\rho_{W^{m-1}}\circ \psi_{W^m}\circ \exp (\tilde B_m)$,
in other words, $\Omega=\exp(\lambda_{m-1}^{-1}(\tilde B_m))$.
By~(\ref{recent2})-(\ref{nado}) along with~(\ref{disks}),
$\Omega\subset \exp(B(2\pi i t_{m-1}, (2/\pi)(r_{m-1}/2)))
\setminus \{e^{2\pi i t_{m-1}}\}$. Now we use the definition of 
the constant $Q_0$, see~(\ref{Q0}), with  
$\epsilon=(2/\pi)(r_{m-1}/2)$. As 
$\epsilon<1/(4q_{m-1}^3)\le 1/(4q)$ and
$q_{m-1}>Q_0$, we therefore have that
$$\exp(B(2\pi i t_{m-1}, \frac{2}{\pi}\frac{r_{m-1}}{2}))\subset 
B(\exp(2\pi i t_{m-1}), \frac{4}{3}\epsilon),$$ 
and since  
$4\epsilon/3<r_{m-1}/2$, we conclude that
the simply-connected domain
$\Omega$ is a subset of
$B(e^{2\pi i t_{m-1}},r_{m-1}/2)\setminus 
\{e^{2\pi i t_{m-1}}\}$. 
It allows us to 
fix a branch $p$ of the function
$v\mapsto (v-e^{2\pi i t_{m-1}})^{1/q_{m-1}}$, which is
defined for $v\in \Omega$.
We thus get a well defined
passage map $\hat p: s\mapsto p\circ \rho_{W^{m-1}}(c_m(s))$
from the local parameter $s\in S_m$ of $O_{m+1}(c)$ to the 
corresponding local parameter in $S_{m-1}$ of the points of $O_m(c)$. 
Now, let $\hat Z_{m}(s):=Z_m(\hat p(s))$, $s\in S_m$, 
be a point of $O_{m}$, for which
$\hat Z_{m}(0)=Z_{m+1}(0)$. For  
$b=\hat Z_m(s)$ of $O_m$, there is the corresponding point 
$b^+=\hat Z^+_m(s)$ of $O_m$, i.e.
$\hat Z^+_m(s)=Z_m(\hat p(s)e_{q_{m-1}})$.
Consider the functions $Z_{m+1}$, $\hat Z_{m}$, and 
$\hat Z^+_m$ holomorphic in $S_m$.
Observe that $Z_{m+1}(s)=\hat Z_{m}(s)$ if and only if $s=0$,
and $\hat Z^+_{m}(s)\not= \hat Z_{m}(s)$ in $S_m$.
Introduce a new function 
$$\zeta_m(s)=\frac{Z_{m+1}(s)-\hat Z_{m}(s)}{\hat Z^+_{m}(s)-\hat Z_{m}(s)}.$$
By the above, $\zeta_m(s)$ obeys the following properties:

(i) it is holomorphic in the disc $S_m$, 

(ii) $\zeta_m(s)\not=1$ in $S_m$,

(iii) $\zeta_m(s)=0$ if and only if $s=0$. 

We conclude that 
\begin{equation}\label{yahas}
|\zeta_m(s)|\le H(\frac{|s|}{v_m}),
\end{equation}
where the function $H$ is defined in Sect.~\ref{H}.

As for the limit parameter $s=s_m$, we get
\begin{equation}\label{yahas1}
|Z_{m+1}(s_m)-\hat Z_m(s_m)|\le H(\frac{|s_m|}{v_m}) 
|\hat Z^+_m(s_m)-\hat Z_m(s_m)|.
\end{equation}

In turn, by ~(\ref{+}),
\begin{equation}\label{++}
|\hat Z^+_m(s_m)-\hat Z_m(s_m)|=
|Z_{m}(s_{m-1} e_{q_{m-1}})-Z_{m}(s_{m-1})|
<3 \frac{2\pi}{q_{m-1}}
\frac{\frac{|s_{m-1}|}{v_{m-1}}}{1-\frac{|s_{m-1}|^2}{v_{m-1}^2}}.
\end{equation}

Here $Z_{m+1}^*:=Z_{m+1}(s_m)$ is {\it any} point of the periodic orbit
$O_{m+1}$ of the limit map $f_{c_*}$ while
$Z_{m}^*:=\hat Z_{m}(s_m)$ is
a point of the periodic orbits $O_{m}$ 
of the same map $f_{c_*}$, which is determined by the $Z_{m+1}^*$.
Because $H$ and $t/(1-t^2)$ are increasing functions, we conclude
from ~(\ref{yahas1}),~(\ref{++}), and~(\ref{v}):
\begin{equation}\label{rel}
|Z_{m+1}^*-Z_{m}^*|\le 3 \frac{2\pi}{q_{m-1}}
\frac{u_{m-1}}{1-u_{m-1}}H(u_m) 
\end{equation}
where 
$u_k=(32 |t_{k+1}| \min\{2n_{k+1}, \frac{2^{n_k/2}}{|p_k q_k|}\})^{1/q_{k}}$, 
$k=0,1,...$.

In turn, the point $Z_m^*$ determines a
point $Z_{m-1}^*$ of the periodic orbit $O_{m-1}$, and so on until
a point $Z_0^*$ of $O_0$.
The inequality ~(\ref{rel}) holds for every $m\ge 1$.
It remains to estimate $|Z_1^*-Z_0^*|$.
Since $Z_0^*$ is a fixed point of $f_{c_*}$, we compare
$|Z_1^*-Z_0^*|$ with $|\beta^*-Z_0^*|$, where $\beta^*$ is the second
fixed point of $f_{c_*}$. By similar considerations, we can write
$$|Z_1^*-Z_0^*|\le H(u_0)|\beta^*-Z_0^*|.$$
On the other hand,
$$|\beta^*-Z_0^*|=|1-\rho^*|$$
where $\rho^*$ is the multiplier of $Z_0^*$. 
Hence,
$$|\beta^*-Z_0^*|\le |1-\exp(2\pi i t_0)|+|\exp(2\pi i t_0)-\rho^*|
\le 2\pi |t_0|+1/(2q_0^3)<7 |t_0|.$$
We have finally, for $m=0,1,2,...$,
\begin{equation}\label{abs}
|Z_{m+1}^*-Z_0^*|< 6 \pi\{|t_0| H(u_0)+
\sum_{k=1}^\infty \frac{1}{q_{k-1}}
\frac{u_{k-1}}{1-u_{k-1}}H(u_k)\}.
\end{equation}
Here $Z_{m+1}^*$ is any point of the periodic orbit $O_{m+1}$
of $f_{c_*}$, for any $m\ge 0$. 
On the other hand, since $\rho^*=2 Z_0^*$ and $|\rho^*|>1-1/(2 q_0^3)>1/2$,
then $|Z_0^*|>1/4$.
Hence, under the condition
~(\ref{exp1}),
for every $m\ge 1$, the periodic orbit
$O_m$ of $f_{c_*}$ lies outside of a fixed neighborhood
$B(0, \delta)$ of zero, where $\delta>1/4-6\pi/100>0$.
It is well known that this implies the
non local connectivity of $J_{c_*}$
(see~\cite{Sor} for a detailed proof of this fact). 
\end{proof}
\subsection{Proof of Theorem~\ref{intronlc}}\label{lcpr}
First, let us make a remark on the condition (S) of Theorem~\ref{intronlc}.
We claim that the value $u_{k, m}$ in $(S)$ can be replaced by
$$\tilde u_{k, m}=|A t_{m+1} 
\max\{B q_k...q_m, \ \frac{\exp(\tilde\gamma q_k...q_{m-1})}
{|p_mq_m|}\}|^{1/q_m},$$ 
for any positive $A$, $B$, and $\tilde \gamma$.
Indeed, given $k$, if $k_1>k$ is big enough and $m\ge k_1$, then
$\tilde u_{k_1, m}/u_{k, m}<1$. Therefore, if~(\ref{Srel0}) holds
and since $H$ is increasing,
then
$$\sum_{m=k_1}^\infty  \frac{\tilde u_{k_1,m}}{q_m(1-\tilde u_{k_1,m})}
H(\tilde u_{k_1,m+1})\le \sum_{m=k_1}^\infty  
\frac{u_{k,m}}{q_m(1-u_{k,m})}H(u_{k,m+1})<\infty.$$
Vice versa, given $k$, 
if $k_1>k$ is big enough and $m\ge k_1$, then
$u_{k_1, m}/\tilde u_{k, m}<1$, i.e., if~(\ref{Srel0}) holds
with $\tilde u_{k, m}$ instead of $u_{k, m}$, then it also holds
with $u_{k_1, m}$. 

Now we can restate  
the conditions (Y0), (Y1), (S) as follows:

$(Y0')_{a-b}$ There is $C>0$, such that, for all $m$ large enough,
\begin{equation}\label{Yrel0'}
|t_m|<\min\{\frac{C}{q_0...q_{m-1}}, 
|p_{m-1}q_{m-1}|\exp(-\frac{q_0...q_{m-2}}{C})\}.
\end{equation}

(Y1') There exist $\beta>0$,  $C_1>0$, such that
\begin{equation}\label{Yrel1''}
\limsup_{m\to \infty}\frac{q_m}
{\max\{(q_0...q_{m-1})^2, \exp(\beta q_0...q_{m-2})\}}>C_1,
\end{equation}

(S') For some $k\ge 0$,
\begin{equation}\label{Srel0'}
\sum_{m=k}^\infty  \frac{\tilde u_{k,m}}{q_m(1-\tilde u_{k,m})}
H(\tilde u_{k,m+1})<\infty, 
\end{equation}
where 
\begin{equation}\label{Rintro0'}
\tilde u_{k,m}=|32 t_{m+1} 
\max\{2 q_k...q_m, \ \frac{2^{q_k...q_{m-1}/2}}{|p_m|q_m}\}|^{1/q_m}, 
m\ge k.
\end{equation}
(We set $q_k...q_{m-1}=1$, if $m=k$.)
It is enough to find a tail $T_{k_0}=\{t_m\}_{m=k_0}^\infty$ 
of the sequence $T_0=\{t_m\}_{m=0}^\infty$, which satisfies
corresponding conditions of Lemma~\ref{notlc} with $n=1$.
Namely,
let us start with the $1$-hyperbolic component $W_0$ 
(the main cardioid) and the
tail $T_{k_0}=\{t_{k_0},t_{k_0+1}...\}$ in place of $W$ and 
$T_0=\{t_0,t_1,...\}$, respectively.
Then we get a sequence of hyperbolic components
$W_0^{k_0,m}$, $m\ge k_0$, where $W_0^{k_0,k_0}=W_0$ and,
for $m>k_0$, the closure of $W_0^{k_0,m}$
touches the closure of $W_0^{k_0,m-1}$ at the point $c_{k_0,m-1}$ 
with internal argument $t_{m-1}$.
By a well known straightening procedure, see~\cite{DH3} and
the proof of Theorem~\ref{intromlc}, 
the following implications hold.
If the sequence of parameters $c_{k_0,m}$ converges,
as $m\to\infty$, to some $c_{k_0, *}$, then the sequence $c_m$ also
converges,
to some $c_\ast$. If, moreover, $M$ is locally connected at $c_{k_0,*}$, then
$M$ is locally connected at $c_\ast$, and if the Julia set $J_{c_{k_0, *}}$ 
is not locally connected, then $J_{c_\ast}$ is not locally connected, too.

Thus to prove the theorem it is enough to find a tail $T_{k_0}$,
such that: (a') the condition $(Y0')_{a-b}$ 
for the whole sequence $T_0$ implies
the condition~(\ref{Yrel1}) for the tail $T_{k_0}$,
(b') the condition (Y1') for $T_0$ implies the condition
~(\ref{Yrel1'}) for $T_{k_0}$, and 
(c') the conditions $(Y0')_{a-b}$ and (S') for $T_0$
implies the condition~(\ref{exp1}) for $T_{k_0}$.
Notice that (a') and (b') are obviously true, for any $k_0$
large enough, because $q_0...q_{k_0-1}\to \infty$ as $k_0\to \infty$.
Let us show (c').

Assume that the condition (S') holds,
for a fixed large enough $k$.
Since $\tilde u_{k,m}$ is decreasing in $k$,
for every $k_0$ large enough,
$$\sum_{m=k_0}^\infty  
\frac{\tilde u_{k_0,m}}{q_m(1-\tilde u_{k_0,m})}H(\tilde u_{k_0,m+1})
<\frac{1}{200}.$$
Thus to satisfy~(\ref{exp1}) for some $T_{k_0}$ instead 
of $T_0$,
it is enough to check that $|t_{k_0}|H(\tilde u_{k_0,k_0})$ tends to zero
as $k_0$ tends to $\infty$. Using again that $\tilde u_{k,m}$ decreases 
as $k$ increases,
it follows from (S'), that, for a fixed $k$ and $k_0\to \infty$,
$$\frac{\tilde u_{k_0-1,k_0-1}}{q_{k_0-1}(1-\tilde u_{k_0-1,k_0-1})}
H(\tilde u_{k_0,k_0})
\le \frac{\tilde u_{k,k_0-1}}{q_{k_0-1}(1-\tilde u_{k,k_0-1})}
H(\tilde u_{k,k_0})\to 0$$
On the other hand, using $(Y0')_{a-b}$
it is easy to see that, as $k_0\to \infty$, 
$$|t_{k_0}|/(\frac{\tilde u_{k_0-1,k_0-1}}
{q_{k_0-1}(1-\tilde u_{k_0-1,k_0-1})})<
\frac{|t_{k_0}|q_{k_0-1}}{\tilde u_{k_0-1,k_0-1}}\to 0.$$
This completes the proof of 
Theorem~\ref{intronlc}.

\section{Non-primitive components}\label{nprim}
\subsection{Proof of Theorem~\ref{univ1}}\label{univ1pr}
Let $W=Z(t_0)$, $Z$ be an $n_0$-hyperbolic component, and 
$t_0=p_0/q_0\not=0$.
By Theorem~\ref{univ}(c),
the function $\psi_Z$ extends to a univalent map
defined in $B(0,1)\cup \exp(B(2\pi i t_0, r_0))$, 
where $r_0=dist(2\pi i t_0,
\partial \tilde \Omega^{\log}_{n_0})$.
We would like to apply Lemma~\ref{q}. The radius $r$ in Lemma~\ref{q}
will be specified as follows: $r=\min\{r_0, r(n_0, p_0/q_0)\}$,
where $r(n_0, p_0/q_0)=\min\{1/(2n_0 q_0^3), |p_0/q_0|2^{-n_0/2}\}$. 
Let us check the conditions of  Lemma~\ref{q} with this specific $r$,
with $X=1$, and with
$Z$, $p_0/q_0$, $W$ in place of $W$, $p/q$, $W(p/q)$
respectively.
Indeed, by the above and by
Lemmas~\ref{r}-\ref{rb}, we have
that $\psi_Z$ has a univalent extension
to $B(0,1)\cup U$, where $U=\exp(B(2\pi i t_0, r))$, and,
furthermore, the domain $V=\psi_Z(U)$ satisfies the conditions
(a)-(b) of Lemma~\ref{q}. Therefore, we indeed can apply Lemma~\ref{q}
with these data.
We conclude, that: 
(i) the function $\rho_W$ extends to a holomorphic function defined in
the domain $\hat V=V\cup W^*\cup Z$, (ii) 
the function $\Psi_{Z, t_0}=\rho_W\circ \psi_Z^{\log}$ is defined
and univalent
in $B(2\pi i t_0, \Lambda_0 r)$, and $2q_0^2/3<|\Psi_{Z, t_0}'|<3q_0^2/2$
in $B(2\pi i t_0, \Lambda_0 r)$, and (iii)
the function $\psi_W$
has a univalent extension to $B(0,1)\cup B(1, 2q_0^2 \Lambda_0 r/3))$
(where $0<\Lambda_0<1$ is a universal constant). 
Now, applying Theorem~\ref{univ}(c) to the 
$n_0 q_0$-hyperbolic component $W$, we get that $\psi_W$
is univalent in $\tilde\Omega_{n_0q_0}$. This along with (iii) gives us that
$\psi_W$ is holomorphic in the domain
$$\Omega:=\tilde\Omega_{n_0q_0}\cup B(0,1)\cup 
B(1, \frac{2}{3}q_0^2 \Lambda_0 r))
=\tilde\Omega_{n_0q_0}\cup B(1, \frac{2}{3}q_0^2 \Lambda_0 r).$$
Let us show that $\psi_W$ is univalent in $\Omega$.
Indeed, by Theorem~\ref{univ}(c), $\psi_W(\tilde\Omega_{n_0q_0})\subset W^*$,
and by (ii)-(iii), 
$$\psi_W(B(1, \frac{2}{3}q_0^2 \Lambda_0 r))=
\psi_Z^{\log}\circ \Psi_{Z, t_0}^{-1}(B(1, \frac{2}{3}q_0^2 \Lambda_0 r))
\subset \psi_Z^{\log}(B(2\pi t_0, \Lambda_0 r)) \subset V,$$
therefore, $\psi_W(\Omega)\subset \hat V$. But (i) tells us that 
$\rho_W$ is well-defined in $\hat V$. Then $\rho_W$ is a
well-defined inverse function to $\psi_W$ in $\Omega$, i.e., 
$\psi_W$ is univalent in $\Omega$.

It remains to check that the domain $\Omega_{n_0, t_0}$ introduced
in Theorem~\ref{univ1} is a subset of $\Omega$.
It is easy to check that, for some $K>0$ and all $n_0, q_0$,
$K^{-1} P<r_0=dist(2\pi i t_0, \partial \tilde \Omega^{\log}_{n_0})<K P,$
where
$$P=\min\{ \frac{n_0 |t_0|}{4^{n_0}}, \frac{t_0^2}{n_0}\}.$$
Obviously, 
$1/(2 n_0 q_0^3)<t_0^2/n_0$ and $n_0 |t_0| 4^{-n_0}<|t_0| 2^{-n_0/2}$.
Combining this, we get that, for some universal $K_1>0$,
$$\frac{2}{3}q_0^2 \Lambda_0 r>K_1q_0^2
\min\{\frac{n_0 |t_0|}{4^{n_0}}, \frac{1}{n_0 q_0^3}\}=
K_1\min\{\frac{n |p_0|}{4^{n_0}}, \frac{1}{n}\}.$$
Taking a bit smaller $0<\tilde K<K_1$
and putting $d=\tilde K\min\{n |p_0| 4^{-n_0}, n^{-1}\}$,
we have that $\exp(B(0, d))\subset B(1, 2q_0^2 \Lambda_0 r/3)$.
Recall that $\Omega_{n_0, t_0}$ 
is defined by its log-projection
as $\tilde\Omega_{n_0q_0}^{\log}\cup B(0, d)$.
Hence,
indeed, $\Omega_{n_0, t_0}\subset \Omega$,
while $\psi_W$ is univalent in $\Omega$.
Finally, it is easy to show that for $q_0$ large enough
$B(0, d)$ covers the portion of the disk
$\{L: |L-R_{n_0q_0}|<R_{n_0q_0}\}$,
which is deleted from $\Omega_{n_0q_0}^{\log}$
(see Theorem~\ref{univ}(c)), hence,
$\tilde\Omega_{n_0q_0}^{\log}$ above can be replaced by
$\Omega_{n_0q_0}^{\log}$. 

\subsection{Proof of Theorem~\ref{limb1}}\label{limb1pr}
The proof consists of a consideration of several cases
using Theorem~\ref{univ1} and Theorem~\ref{univ}.
Note that in some cases we get much better bounds
for the size of the limb $L(W, p/q)$. Keeping the notations of 
Subsection~\ref{univ1pr},
$W=Z(p_0/q_0)$, where $Z$ is an $n_0$-component, and so 
$W$ is an $n$-component, where $n=n_0 q_0$.
Denote by $D$ the diameter of $L(W, p/q)$.

If $q\le 8^n$, the bound holds trivially, with $\tilde A=4$.
So, in what follows, $q>8^n$. 

Assume $\frac{n |p|}{q}\ge \frac{d}{1000}$, where 
$d=\tilde K \min\{ \frac{n |p_0|}{4^{n_0}}, \frac{1}{n}\}$
is taken from Theorem~\ref{univ1}. 
If $n \frac{|p|}{q}\ge \frac{\tilde K}{1000 n}$, then, by Theorem~\ref{limb},
$D\le A \frac{4^n}{|p|}\le A (1000 n^2 \tilde K^{-1}) \frac{4^n}{q}\le 
\tilde A \frac{8^n}{q}$, 
for
some $\tilde A$ independent on $n, q$. If
$n \frac{|p|}{q}\ge \frac{\tilde K n |p_0|}{1000 \ 4^{n_0}}$, then similarly
$$D\le A \frac{4^n}{|p|}\le 1000 A \frac{4^{n_0+n}}{\tilde K |p_0|q}\le 
\tilde A \frac{8^n}{q},$$ 
for
some $\tilde A$ independent on $n, q$.
Hence, one can assume that 
\begin{equation}\label{sm}
n |t|=n \frac{|p|}{q}<\frac{d}{1000}.
\end{equation}
By Theorem~\ref{univ1} and Koebe, for every $L\in B(0, d/2)$,
\begin{equation}\label{111}
|(\psi_W^{\log})'(L)|\le 81|(\psi_W^{\log})'(0)|.
\end{equation}
In turn, by Theorem~\ref{univ}(a) and~(\ref{gu}),
\begin{equation}\label{222}
|(\psi_W^{\log})'(0)|=q_0^{-2} |\psi_Z'(\exp(2\pi i t_0))|\le
q_0^{-2} \frac{B_0 4^{n_0}(1+o(1))}{n_0 |\exp(2\pi i t_0)-1|}\le
q_0^{-2} \frac{B_1 4^{n_0}}{n_0 |p_0/q_0|}<\frac{B_1 4^{n_0}}{n},
\end{equation}
for some absolute constant $B_1>0$.
Let us show that Yoccoz's circle 
$Y_n(t)=\{L: |L-(2\pi it+\frac{n\log 2}{q})|<\frac{n\log 2}{q}\}$
is covered by $B(0, d/2)$. 
First, $2n\log 2/q<2n|t|<d/2$, by~(\ref{sm}). 
Hence, it is enough to check that
$(2\pi t)^2+(n\log 2/q)^2<(d/2-n\log 2/q)^2$, or
$(2\pi t)^2<(d/2)^2(1-4n\log 2/(q d))$, and this holds by~(\ref{sm}).
Thus $Y_n(t)\subset B(0, d/2)$. Therefore, by~(\ref{111})-(\ref{222}),
$$D<2 n (\log 2/q) 81 B_1 \frac{4^{n_0}}{n}<200 B_1 \frac{8^n}{q}.$$
\section{Appendix: geometric bounds for the multiplier}\label{peter}
\subsection{Periodic points on the boundary of a basin of attraction}
\label{apcon}
Let us fix an $n$-hyperbolic component $W$. For
$c\in W$, the map $f_c$ has the attracting periodic orbit $O(c)$ 
of period $n$, and $\rho_W(c)$
denotes its multiplier. Given a rational number $p/q\not=0$,
consider the point $c(W, p/q)$ of $\partial W$ with the internal argument
$p/q$. As we know, see the beginning of Section~\ref{bif},
for $c$ near $c(W, p/q)$, there is a unique periodic
orbit $O_c^{p/q}$ of $f_c$ of period $nq$, which collides with $O(c)$
as $c\to c(W, p/q)$. 
Its multiplier $\rho_{W(p/q)}(c)$ extends to a function,
which is defined and holomorphic in the union $W\cup \hat B\cup W(p/q)$,
where $\hat B$ is a small neighborhood of $c(W, p/q)$.
Here we prove the inequality (\ref{inside}) of Section~\ref{bif}:
\begin{theo}
For $c\in W$, such that $\rho_W(c)\not=0$,
\begin{equation}\label{insideap}
\frac{|\log\rho_{W(p/q)}(c)|^2}{\log|\rho_{W(p/q)}(c)|}
<q^2\frac{|\log\rho_W(c)-2\pi i p/q|^2}
{-\log|\rho_W(c)|},
\end{equation}
for some branch of $\log\rho_{W(p/q)}(c)$ and any branch of $\log\rho_W(c)$.
\end{theo}
\begin{proof}
This follows from Theorem A of~\cite{pet1} and 
Theorem 2 of~\cite{levcoll} (see also~\cite{Pe}).
Let us fix $c\in W$ with
$\rho_W(c)\not=0$, and consider the dynamical plane of $f_c$.
Denote by $\Omega$ the component of the immediate
basin of attraction of $O(c)$, which contains the critical point $0$
and a point $b\in O(c)$. 
A Riemann map $R: \Omega\to B(0,1)$ normalized by 
$R(b)=0$ and with an appropriate $arg R'(b)$
conjugates $f_c^n: \Omega\to \Omega$ to the
Blaschke product $B_\rho(z)=z(z+\rho)/(1+\bar\rho z)$,
where $\rho=\rho_W(c)$. The Julia set of $B_\rho$ is the unit circle $S^1$.
Given a rational number $p/q\not=0$, the map $B_\rho:S^1\to S^1$
has a unique rotation set $\Delta(p/q)$ with the rotation number $p/q$.
(This means that the restriction $B_\rho: \Delta(p/q)\to \Delta(p/q)$
can be lifted and extended to an increasing continuous map 
$\tilde B:{\bf R}\to {\bf R}$, such that $\tilde B(x+1)=\tilde B(x)+1$,
which then defines in a usual manner the rotation number $p/q$, 
see e.g.~\cite{BS}.)
The set $\Delta(p/q)$ is a repelling periodic orbit of $B_\rho$
of period $q$. 
Let $\Lambda$ be the multiplier of $\Delta(p/q)$, so that $\Lambda>1$.
The following bound is proved in~\cite{pet1}, Theorem A.
For every branch $L$ of the logarithm $\log\rho$,
\begin{equation}\label{blyabound}
0<\log\Lambda<q^2\frac{|L-2\pi i p/q|^2}{2|Re(L)|}.
\end{equation}
The inverse map $R^{-1}: B(0,1)\to \Omega$ has a radial limit
denoted by $R^{-1}(w)$
at every repelling periodic point $w\in S^1$ of $B_\rho$, 
and the point $R^{-1}(w)$ is either 
repelling or parabolic periodic point of 
$f_c^n:\partial \Omega\to \partial \Omega$, see~\cite{pomm},~\cite{Pe}.
(Alernatively, here and below one could use the fact that the 
boundary $\partial\Omega$ is locally connected, because $f_c$ is
a hyperbolic polynomial; in turn, this implies that $R^{-1}$ has
a homeomorphic continuation to the boundary $\partial B(0,1)$.)
The set $\tilde O_c(p/q)=\{R^{-1}(w): w\in \Delta(p/q)\}$
is a periodic orbit
of $f_c^n$ of period $q$, which lies in $\partial \Omega$
and is repelling, because $f_c$ is hyperbolic. Denote
by $\tilde \rho_c$ its multiplier. Let $w_0\in \Delta(p/q)$ and
$z_0=R^{-1}(w_0)$. 
Now we proceed as in~\cite{levcoll}. Linearizing $B^q_\rho$
near $w_0$ by a conformal map $g$, $g(0)=w_0$, 
we find a map $h=R^{-1}\circ g$, which is conformal near $0$ and
such that $h$ maps a small semidisk
$D(\epsilon)=\{w: |w|<\epsilon, Im(w)>0\}$ into $\Omega$
and conjugates $w\mapsto \Lambda w$ with 
$f_c^{nq}:\Omega\to \Omega$ wherever it makes sense.
Let us fix also a small $r>0$ so that $f_c^{nq}$ is linearizable 
in $B(z_0, r)$
and let $\mu: B(z_0, r)\to {\bf C}$, $\mu(z_0)=0$, be a conformal map,
which conjugates $f_c^{nq}$ with its linear part $z\mapsto \tilde \rho_c z$.
As $R^{-1}$ has the radial limit $z_0$ at $w_0$, by Lindelof' theorem,
$R^{-1}(w)\to z_0$ uniformly in any Stolz angle at $w_0$. 
Hence, $h(w)\to z_0$ uniformly in
$\{|w|<\epsilon, arg(w)\in (t, \pi-t)\}$, for any $t\in (0,\pi/2)$.
Thus, for any $t\in (0,\pi/2)$ there is $\epsilon_t>0$, such that
$h(D(\epsilon_t, t))\subset B(z_0, r)$, where
$D(\epsilon_t, t)=\{|w|<\epsilon_t, arg(w)\in (t, \pi-t)\}$.
Set $U=\cup_{t}D(\epsilon_t, t)$ and 
$V=\tilde h(U)$,
where $\tilde h=\mu\circ h$.
Then $U$ and $V$ are topological disks, such that 
$0\in \partial U$, $U\subset \Lambda U:=\{\Lambda z: z\in U\}$,
$\cup_{k=0}^\infty \Lambda^k U=\{z: Im(z)>0\}$, and
$0\in \partial V$, $V\subset \tilde \rho_c V$, and 
$\tilde h:\Lambda U\to \tilde \rho_c V$ 
is a conformal homeomorphism
that conjugates the linear maps
$z\mapsto \Lambda z$ and $z\mapsto \tilde\rho_c z$
on $U$. 
This is the framework of 
Theorem 2 of \cite{levcoll}. It states that under these conditions
\begin{equation}\label{lbound}
\frac{|\log \tilde \rho_c|^2}{\log |\tilde \rho_c|}
<2\log\Lambda,
\end{equation}
for some branch $\log \tilde \rho_c$.
Combining (\ref{blyabound}) and (\ref{lbound}), we get
\begin{equation}\label{blbound}
\frac{|\log \tilde \rho_c|^2}{\log |\tilde \rho_c|}
<q^2\frac{|\log\rho_W(c)-2\pi i p/q|^2}{-\log|\rho_W(c)|},
\end{equation}
for some branch $\log \tilde \rho_c$ and any branch $\log\rho_W(c)$.
Now, let us show that, when $c$ is close to $c(W, p/q)$,
then the periodic orbit $\tilde O_c(p/q)$ 
is just the periodic orbit $O^{p/q}(c)$.
Indeed, let $c\to c(W, p/q)$ radially, i.e. along the
curve $\gamma=\{\rho_W^{-1}(r e^{2\pi i p/q}), 0<r<1\}$.
Then there is a branch  of $\log\rho_W(c)$ converging to
$2\pi i p/q$ as $c\to c(W, p/q)$ and, 
by~(\ref{blbound}) we conclude that $\log \tilde\rho_c\to 0$
as $c\to c(W, p/q)$.
Since every periodic orbit of $f_{c(W, p/q)}$ other than
$O(c(W, p/q))$ is repelling, it follows by continuity, that
$\tilde O_c(p/q)$ collides with $O(c)$ as $c\to c(W, p/q)$
along $\gamma$. On the other hand, $O^{p/q}(c)$ is the only periodic
orbit with this property. Hence,  $\tilde O_c(p/q)$ and $O^{p/q}(c)$ 
coincide for $c\in \gamma$ close to $c(W, p/q)$, therefore,
they coincide for every $c\in W$ close to $c(W, p/q)$.
Thus their multipliers are equal, i.e., for every $c\in W$, such that
$\rho_W(c)\not=0$,
$\tilde \rho_c$ in~(\ref{blbound}) can be replaced
by $\rho_{W(p/q)}(c)$.
\end{proof} 
\subsection{Disconnected Julia set}\label{apdis}
Suppose $c\in {\bf C}\setminus M$, i.e.,
$J_c$ is totally disconnected. 
Let $B_c$ be the Bottcher coordinate function
for $f_c$ at infinity, i.e., $B_c$ is defined and univalent
in a neighborhood of infinity, such that $B_c(z)/z\to 1$ as $z\to \infty$
and $B_c(f_c(z))=[B_c(z)]^2$. The function
$G_c(z)=\log |B_c(z)|=\lim_{n\to \infty}\frac{1}{2^n}\log|f_c^n(z)|$
extends to a harmonic function defined in the basin of infinity
$A_c={\bf C}\setminus J_c$, such that $G_c(z)\to 0$ as 
$z\to \partial A_c$. 
In turn, $B_c$ extends to a univalent map defined in the domain
$\{z: G_c(z)>G_c(0)\}$. In particular, the value $B_c(c)$ is well-defined.
(By~\cite{DH1}, the map
$c\mapsto B_c(c)$ is a holomorphic
isomorphism of the complement of $M$ onto the complement
of $\overline{B(0,1)}$.)
Let us introduce parameters $a_c>0$ and $t_c\in [0,1)$ such that
$$B_c(c)=\exp(2 a_c+2\pi i t_c).$$
Note that $a_c=G_c(0)$ and $t_c$ is the argument of the 
parameter ray to $M$, which passes through $c\in {\bf C}\setminus M$. 

Let $z$ be a periodic point of $f_c$ of period $n$,
and $\rho=(f_c^n)'(z)$ its multiplier.

If one puts in~\cite{EL}, Theorem 1.6,
$d=2$, $a=b=a_c$ and $k=0$, we get:
\begin{theo}\label{disbelow}
$$\frac{1}{n}\log |\rho|\ge a_c.$$
\end{theo}
Now, we give a bound for $\log\rho$ that involves $a_c$, $t_c$,
and a rotation number of $z$,
which is similar to the Yoccoz bound~(\ref{yoccircle})
though holds for non-connected Julia sets.
We consider dynamical rays to the disconnected set $J_c$
(see e.g.~\cite{levprz}). Each ray is either
an unbounded smooth curve, which crosses
every level curve
$\{z: G_c(z)=r\}$ orthogonally and terminates in $J_c$, or
a one-sided limit of such smooth rays. In the latter case, 
the ray is called non-smooth,
or left (right) if it is a limit of smooth rays from the
left (right).  A ray is non-smooth if and only if it contains
a critical point of the function $G_c$, i.e.,
a preimage by some $f_c^k$, $k\ge 0$, of the critical point $0\in A_\infty$.
Every ray has a well-defined angle (or argument) though
some angle may correspond to two non-smooth rays, which are left and 
right limits of smooth ones.  
Denote by $\Lambda(z)$
the set of angles of dynamical rays that land at $z$.
The following is proved in~\cite{levprz}.
The set $\Lambda(z)$ is a non-empty closed nowhere dense subset of the 
unit circle $S^1={\bf R}/{\bf Z}$.
It is finite if and only if it contains
a rational angle $t_0$. In this case, $t_0$ is a periodic point
under the doubling map $\sigma: t\mapsto 2t(mod \ 1)$ of some period $nq$,
$q\ge 1$,
and every other point $t\in \Lambda(z)$ is periodic by $\sigma$ of the
same period. The rotation number $p/q$ of $z$ 
is the order at which $f_c^n$ permutes 
(locally) each cycle of $q$ dynamical rays that land at $z$.
Let us assume that the rotation number of $z$ is $p/q$,
and choose $t_0\in \Lambda(z)$. Consider the following periodic orbit
of the map $\sigma^n$:
$\Lambda_{t_0}=\{\sigma^{kn}(t_0): k=0,1,...,q-1\}$. It is a subset
of $\Lambda(z)$. We associate to $\Lambda_{t_0}$ 
an angle $\alpha\in (0, \pi)$
as follows~\cite{Leyo}. First, the map $\sigma$ has a natural
extension to 
a half-cylinder
$\tilde S=\{(x, y): x\in S^1, y\ge 0 \}$ by $\sigma(x,y)=(2x(mod \ 1), 2y)$.
Consider the set $\{N_w\}$ of all vertical segments
in $\tilde S$ with top points $w$,
over all $w$ such that $\sigma^j(w)=w_c:=(i a_c/\pi, t_c)$,
for some $1\le j\le n$. Each segment $N_w$ is a vertical
``needle'' with the top $w$ and a base point $x_w\in S^1$.
Given $N_w$, let $x_l$ be the point of the set $\Lambda_{t_0}$,
which is the closest to the base point $x_w$, $x_w\not=x_l$,
from the left 
(measured in $S^1$). Consider the triangle with the
vertices $w$, $x_w$, and $x_l$.
Let $\alpha_l(N_w)$ be its angle at the vertex $w$.
Define $\alpha_l$ as the minimum of the $\alpha_l(N_w)$, over
all $N_w$. The angle $\alpha_r$ is defined similarly 
using the points from the right of $x_w$. Then the angle $\alpha$
is said to be $\alpha_l+\alpha_r$, if the ray of argument $t_0$
is smooth, and otherwise either 
$\alpha_l$ or $\alpha_r$, depending on whether left or right 
ray of the angle $t_0$ lands at $z$.
We have the following generalization of~(\ref{yoccircle}):
\begin{theo}\label{yoccircledis} (see~\cite{Leyo})
A branch of $\log\rho$ is 
contained in the disk
$$
Y_n(p/q, \alpha)=\{L: |L-(2\pi i\frac{p}{q}+\frac{n\pi\log 2}{q\alpha})|<
\frac{n\pi\log 2}{q\alpha}\},$$
\end{theo}
We apply this bound when $p/q=0/1$, and $t_0$ (a periodic point of
$\sigma$ of period $n$) is the argument of a non-smooth, say, left, ray. 
In particular, $t_c\in \Lambda_{t_0}$.
In this case, 
$\alpha=\alpha_l\ge \delta(n, a_c)$, where
\begin{equation}\label{anglezero}
\delta(n, a)=\arctan\frac{\pi}{(2^n-1)a}.
\end{equation}
Indeed, we check, that, for every ``needle'' $N_w$,
\begin{equation}\label{ochennado}
\alpha_l(N_w)\ge \delta(n, a_c).
\end{equation}
Let $\sigma^j(w)=w_c$, $1\le j\le n$. Denote by $s$ 
the distance between the points
$x_w$, $x_l$ (measured on $S^1={\bf R}/{\bf Z}$). 
Let us show that $s\ge 1/(2^j(2^n-1))$. Indeed, otherwise
$2^j s<1$, hence,
$s=s_0/2^j$, where $s_0$ is the distance between
$\sigma^j(x_l)$ and $t_c=\sigma^j(x_w)$.
In the considered case, $\sigma^j(x_l)$ and $t_c$ are different points
of the periodic orbit $\Lambda_{t_0}$. Since 
the distance between any two points of $\Lambda_{t_0}$
is at least $1/(2^n-1)$, then $s\ge 1/(2^j(2^n-1))$. 
Thus, in any case,
$s\ge 1/(2^j(2^n-1))$, and then 
$$\tan \alpha_l(N_w)=\frac{s}{(a_c/\pi)/2^j}
\ge \frac{1/(2^j(2^n-1))}{(a_c/\pi)/2^j}=
\tan \delta(n, a_c).$$
Therefore,~(\ref{ochennado}) is checked.
Under this setting,
Theorem~\ref{disbelow} and Theorem~\ref{yoccircledis} imply immediately:
\begin{coro}\label{zero}
Assume $t_c$ is a periodic point of $\sigma: S^1\to S^1$ of period $n$,
and a (one-sided) dynamical ray of $f_c$ of angle $t_c$ lands at a periodic
point of $f_c$ of period $n$ with multiplier $\rho$. Then,
for a branch of $\log\rho$,
\begin{equation}\label{diszero}
\log\rho\in Y_n(0, \delta(n, \frac{1}{n}\log|\rho|))=
\{L: |L-R_{n, |\rho|}| < R_{n, |\rho|} \},
\end{equation}
where 
\begin{equation}\label{rnrho}
R_{n, |\rho|}=\frac{n\pi\log 2}{\arctan\frac{n\pi}{(2^n-1)\log|\rho|}}.
\end{equation}
\end{coro}
Note that~(\ref{diszero}) is proved in~\cite{LeSo} as well.

\end{document}